\documentclass[draft]{amsart}
\usepackage{a4wide}
\usepackage{amsfonts}
\usepackage{amssymb}
\usepackage{latexsym}
\usepackage{caption}
\usepackage{color}
\usepackage{euscript}
\usepackage{mathrsfs} 
\usepackage{dsfont}
\usepackage{bbold}
\usepackage{graphicx}\usepackage[all]{xy}
\usepackage[T1]{fontenc}
\usepackage{setspace}
\usepackage{pgf,tikz}
\usetikzlibrary{arrows}
\definecolor{qqqqcc}{rgb}{0.0,0.0,0.8}
\definecolor{qqqqff}{rgb}{0.0,0.0,1.0}
\definecolor{uuuuuu}{rgb}{0.26666666666666666,0.26666666666666666,0.26666666666666666}

\makeatletter
\@namedef{subjclassname@2010}{%
\textup{2010} Mathematics Subject Classification}
\makeatother

\theoremstyle{plain}
\newtheorem{thm}{Theorem}
\newtheorem{prop}[thm]{Proposition}
\newtheorem{cor}[thm]{Corollary}
\newtheorem{lem}[thm]{Lemma}

\theoremstyle{definition}
\newtheorem{dfn}[thm]{{\it Definition}}
\newtheorem{exa}[thm]{{\it Example}}
\theoremstyle{remark}
\newtheorem{rem}[thm]{{\it Remark}}
\newtheorem*{rem*}{{\it Remark}}

\DeclareMathOperator{\Alg}{{\mathrm{Alg}}}
\DeclareMathOperator{\dess}{{\mathsf{Des}}}

\DeclareMathOperator{\dzii}{{\mathsf{Chi}}}

\DeclareMathOperator{\I}{I}

\DeclareMathOperator{\koo}{{\mathsf{root}}}

\DeclareMathOperator{\Lat}{{\mathsf{Lat}}}
\DeclareMathOperator{\lin}{\mathsf{lin}}
\DeclareMathOperator{\paa}{{\mathsf{par}}}

\DeclareMathOperator{\trig}{\mathbb{C}_T[X]}
\DeclareMathOperator{\unitT}{\mathcal{T}}
\DeclareMathOperator{\unitL}{\mathcal{L}}
\newcommand*{\card}[1]{\mathrm{card}(#1)}

\newcommand*{\des}[1]{{\dess(#1)}}
\newcommand*{\dzD}[1]{{\EuScript D}\big(#1\big)}
\newcommand*{\dz}[1]{{\EuScript D}(#1)}
\newcommand*{\dzi}[1]{\dzii(#1)}
\newcommand*{\dzin}[2]{\dzii^{\langle#1\rangle}(#2)}

\newcommand*{\lambdab}{{\boldsymbol\lambda}}

\newcommand{\gphi}{\varGamma_{\hat\varphi}}

\newcommand*{\hphi}{\hat\varphi}
\newcommand*{\hphin}{\hat\varphi_n}
\newcommand*{\hphiw}{\hat\varphi_w}
\newcommand*{\hpsi}{\hat\psi}
\newcommand*{\mphi}{M_{\hat\varphi}}
\newcommand*{\mphiw}{M_{\hat\varphi_w}}
\newcommand*{\mphin}{M_{\hat\varphi_n}}
\newcommand*{\mpsi}{M_{\hat\psi}}
\newcommand*{\multiskal}{\mathcal M}
\newcommand*{\pa}[1]{\paa(#1)}
\newcommand*{\pan}[2]{\paa^{#1}(#2)}

\newcommand*{\slam}{S_{\boldsymbol \lambda}}

\newcommand*{\smalloplus}{\raise0pt\hbox{$\scriptscriptstyle \oplus$}}

\newcommand*{\tcal}{{\mathscr T}}
\newcommand*{\gcal}{{\mathscr S}}

\newcommand*{\nul}[1]{\mathcal{N}(#1)}
\newcommand*{\toe}{\overset{\ee}{\longrightarrow}}
\newcommand*{\toh}{\overset{\calh}{\longrightarrow}}
\newcommand*{\tosot}{\overset{\mathrm{SOT}}{\longrightarrow}}
\newcommand*{\pth}{\mathscr{P}}

\newcommand*{\multi}{\mathcal{M}(\lambdab)}

\def\pp{\EuScript{P}}

\newcommand{\Ge}{\geqslant}

\def\is#1#2{\langle#1,#2\rangle}
\def\bis#1#2{\big\langle#1,#2\big \rangle}
\def\Bis#1#2{\Big\langle#1,#2\Big \rangle}

\DeclareMathOperator{\D}{d\!}

\def\Dbb{\mathbb D}
\def\N{\mathbb N}
\def\Z{\mathbb{Z}}
\def\R{\mathbb{R}}
\def\C{\mathbb{C}}
\def\T{\mathbb{T}}

\def\calh{\mathcal H}
\def\calm{\mathcal {GM}}

\def\cala{\mathcal A}

\def\bsb{{\mathbf B}}

\def\aa{\EuScript A}
\def\ee{\EuScript E}
\def\hh{\EuScript H}

\def\mm{\EuScript M}
\def\ff{\boldsymbol f}
\def\gg{\boldsymbol g}

\DeclareMathAlphabet{\mathpzc}{OT1}{pzc}{m}{it}

\usepackage{scalerel}

\newcommand\reallywidehat[1]{\arraycolsep=0pt\relax%
\begin{array}{c}
\stretchto{
  \scaleto{
    \scalerel*[\widthof{\ensuremath{#1}}]{\kern-.5pt\bigwedge\kern-.5pt}
    {\rule[-\textheight/2]{1ex}{\textheight}} 
  }{\textheight} %
}{0.5ex}\\           
#1\\                 
\rule{-1ex}{0ex}
\end{array}
}

\title[Generalized multipliers: commutant and reflexivity]{Generalized multipliers for left-invertible analytic operators and its applications to commutant and reflexivity}
\author{P.\ Dymek}
\author{A.\ P{\l}aneta}
\author{M.\ Ptak}
\address{Katedra Zastosowa\'n Matematyki, Uniwersytet Rolniczy w Krakowie, ul. Balicka 253c, 30-198 Krak\'ow, Poland}
    \email{piotr.dymek@urk.edu.pl}
    \email{artur.planeta@urk.edu.pl}
    \email{rmptak@cyf-kr.edu.pl}
   
\begin{document}
\setstretch{1.2}

\begin{abstract}
We introduce generalized multipliers for left-invertible analytic operators. We show that they form a Banach algebra and characterize the commutant of such operators in its terms. In the special case, we describe the commutant of balanced weighted shift only in terms of its weights. In addition, we prove two independent criteria for reflexivity of weighted shifts on directed trees.
\end{abstract}
\maketitle
\section{Introduction}
We study left-invertible analytic operators using analytic function theory approach, which was initiated by Shimorin in \cite{shim}. This class of operators is quite large, it contains for example shifts on generalized Dirichlet spaces (see \cite{ric}), shifts on weighted Bergmann space with logarythmically subharmonic weights on the unit disc in the complex plane (see \cite{shi} and \cite{hed}) and left-invertible weighted shifts on leafless and rooted directed trees (see \cite{j-j-s-2012-mams} and \cite[Lemma 3.3]{c-t}). 

 A characterization of the commutant of a given operator is one of the way of investigation of the operator itself, see \cite[\S 12]{con}. The classical result on unilateral shift of  (the multiplication by the independent variable on the Hardy space $H^2$)  says that its commutant  is the algebra of all multiplications by bounded analytic functions, see \cite[\S 26]{con}. In the case of unilateral shift of arbitrary multiplicity, its commutant is the algebra of all bounded  analytic operator--valued functions (see \cite{hal,he-lo,r-r}). It was shown by Shields in \cite{shi}, that the commutant of unilateral weighted shift of multiplicity one may be identified with the algebra of its multipliers (some specific formal power series generating bounded analytic functions). On the other hand, the multipliers for weighted shifts on rooted directed trees, introduced in \cite{b-d-p-2015},  are not sufficiently large to determine the whole commutant of the operator. Hence, our motivation was to generalize the notion of multipliers to more general context. Our approach is based on the Cauchy type multiplication of the formal power series, which appeared in papers (see \cite{gel, shi, j-l}) by Shields (for classical weighted shifts), Gellar (for general bilateral shifts), Jewell and Lubin (for commuting $n$-tuple of unilateral shifts).

In section 3, we use Shimorin's model for left-invertible analytic operator $T$. This enables us to consider such operator as a multiplication operator by an independent variable on a space of analytic functions with values in $\nul{T^*}$ - the kernel of the adjoint of $T$. We define generalized multipliers for $T$, whose coefficients are bounded operators on $\nul{T^*}$ and denote the set of all generalized multipliers by $\calm(T)$. Moreover, we distinguish the set $\multiskal(T)$ of all generalized multipliers, whose coefficients are a scalar multiple of the identity operator. We prove that both spaces $\calm(T)$ and $\multiskal(T)$ are a Banach algebra (see Theorem \ref{cauchym}). In addition, the space $\calm(T)$ is unitary equivalent to the commutant of $T$ (see Theorem \ref{komutant}).  

In section 4, we apply general theory for the weighted shift $\slam$ on rooted directed tree, which was inspired in \cite{c-t}. Using separated bases we define the rotation of the vector in the space of analytic functions with values in $\nul{\slam^*}$. This enables us to prove that truncated multipliers tends to the original multiplier as the size of truncation grows to infinity. As a consequence, we obtain that multipliers for $\slam$ defined in \cite{b-d-p-2015} and $\multiskal(\slam)$ are in fact the same object (see Proposition \ref{skalarne}). Hence, the notion of generalized multipliers extends previously defined multipliers. Using the aforementioned results, in section 5, in Theorem \ref{kom}, we describe the commutant of balanced weighted shift on rooted directed tree (whose kernel of the adjoint is finite dimensional) only in terms of weights of the operator. In section 6, we return to general context and obtain a criterion for reflexivity of the left-invertible analytic operator (see Theorem \ref{T-ref}). In addition, for the case of weighted shifts on directed trees we prove a second criterion for reflexivity, which is independent from the previous one and it is a  generalization of \cite[Theorem 4.3]{b-d-p-p-2017}. 

 Recently, the commutant and reflexivity for $n$-tuples of multiplication operators by independent variables $z_1,\ldots,z_n$ on a reproducing Hilbert spaces of $E$-valued holomorphic on bounded domain in $\C^n$ admitting polynomial approximation, where $E$ is a separable Hilbert space, were studied in the paper \cite{c-p-t-2}. In particular, some new results for commutant of Bergmann weighted shifts (which are automatically balanced) on rooted, locally finite, leafless directed trees with finite branching index (with finite dimensional kernel of the adjoint) were proven.
\section{Preliminaries}
Let $\N$, $\R$ and $\C$ denote the set of all natural numbers, real numbers and complex numbers, respectively. Set $\N_0=\N\cup\{0\}$. Denote by $\T$ the unit circle $\{z\in\C\colon |z|=1\}$ and by $\Dbb_r$ the open unit disc $\{z\in\C\colon |z|<r\}$ of radius $r>0$. By $\chi_\sigma$ we denote the characteristic function of the set $\sigma$. In all what follows we use the convention $\sum_{i\in \emptyset}x_i=0$ and $\prod_{i\in \emptyset}x_i=1$. Given two sets $X,Y$, the symbol $Y^X$ stands for the set of all functions $f\colon X \to Y$.

Let $\hh$ be a complex Hilbert space. If $A$ is a (linear) operator in $\hh$, then $\dz{A}$ and $A^*$ denote the domain and the adjoint of $A$, respectively (in case it exists). We write $\bsb(\hh)$ for the algebra of all bounded operators on $\hh$ equipped with the standard operator norm.  If $A\in\bsb(\hh)$, then $r(A)$ denotes the spectral radius of $A$. Let $\mathcal{W}$ be a subalgebra of $\bsb(\hh)$. Then $\Lat \mathcal{W}$ stands for the set of all invariant closed subspaces of all operators $A\in\mathcal{W}$; recall that a closed subspace $\mathcal{L}$ of $\hh$ is \textit{invariant} for $A\in\bsb(\hh)$ if $A\mathcal{L}\subset \mathcal{L}$. If $\mm$ is a family of subspaces of $\hh$, then we set $\Alg \mm = \{ A \in \bsb(\hh)\colon A\mathcal{L}\subset \mathcal{L} \text{ for every } \mathcal{L} \in \mm\}$. The algebra $\mathcal{W}$ is said to be {\em reflexive} if $\Alg \Lat \mathcal{W} = \mathcal{W}$. Given $A\in\bsb(\hh)$, the symbol $\mathcal{W}(A)$ stands for the smallest algebra containing $A$ and the identity operator $\I_\hh$ and closed in the weak operator topology; if $\mathcal{W}(A)$ is reflexive, then $A$ is said to be {\em reflexive}. Note that $\Lat A=\Lat \mathcal{W}(A)$. For the first time the reflexive operators were investigated in
 \cite{Sa}.


Let $V$ be a nonempty set. Then $\ell^2(V)$ denotes the Hilbert space of all functions $f\colon V\to\C$ such that $\sum_{v\in V}|f(v)|^2<\infty$ with the inner product given by $\is{f}{g}=\sum_{v\in V} f(v)\overline{g(v)}$ for $f,g\in \ell^2(V)$. The norm induced by $\is{\cdot}{-}$ is denoted by $\| \cdot \|$. For $u \in V$, we define $e_u \in \ell^2(V)$ to be the characteristic function of the one-point set $\{u\}$; clearly, $\{e_u\}_{u\in V}$ is an orthonormal basis of $\ell^2(V)$. Given a subset $W$ of $V$, $\ell^2(W)$ stands for the subspace of $\ell^2(V)$ composed of all functions $f$ such that $f(v)=0$ for all $v\in V\setminus W$. By $P_{W}$ we denote the orthogonal projection from $\ell^2(V)$ onto $\ell^2(W)$. 

Let $\tcal=(V,E)$ be a directed tree ($V$ and $E$ stand for the sets of vertices and directed edges of $\tcal$, respectively). Denote by $\paa$ the partial function from $V$ to $V$ which assigns to a vertex $u\in V$ its parent $\pa{u}$ (i.e.\ a unique $v \in V$ such that $(v,u)\in E$). For $k\in \N$, $\paa^k$ denotes the $k$-fold composition of the partial function $\paa$; $\paa^0$ denotes the identity map on $V$. Set $\dzin n u= \{v\in V\colon \paa^n(v) =u\}$ for $u \in V$, $n\in\N_0$ and $\des{u} = \bigcup_{n=0}^\infty \dzin n u $. A vertex $u \in V$ is called a {\em root} of $\tcal$ if $u$ has no parent. A root is unique (provided it exists); we denote it by $\koo$. The tree $\tcal$ is {\em rooted} if the root exists. The tree $\tcal$ is {\em leafless} if $\card{\dzi{v}}\Ge 1$ for every $v\in V$, where $\card{Y}$ denotes the cardinal number of the set $Y$. Suppose $\tcal$ is rooted. We set $V^\circ=V\setminus \{\koo\}$. If $v\in V$, then $|v|$ denotes the unique $k\in\N_0$ such that $\paa^k(v)=\koo$. 
A subgraph $\gcal$ of $\tcal$ which is a directed tree itself is called a {\em subtree} of $\tcal$. A {\em path} in $\tcal$ is a subtree $\pth=(V_\pth,E_\pth)$ of $\tcal$ which satisfies the following two conditions: (i) $\koo\in\pth$, (ii) for every $v\in V_\pth$, $\card{\mathsf{Chi}_{\pth}(v)}=1$. The collection of all paths in $\tcal$ is denoted by $\pp=\pp(\tcal)$. We refer the reader to \cite{j-j-s-2012-mams} for more information on directed trees.

\vskip 3mm
{\bf Caution}: {\em All the directed trees considered here are assumed to be rooted and countably infinite.}

\vskip 3mm
Weighted shifts on  directed trees are defined as follows (see \cite{j-j-s-2012-mams}). Let $\tcal=(V,E)$ be a directed tree and let $\lambdab=\{\lambda_v\}_{v \in V^{\circ}} \subseteq \C$. We define the map  $\varLambda_\tcal^\lambdab\colon \C^V\to \C^V$ via
   \begin{align*}
(\varLambda_\tcal^\lambdab f) (v) =
   \begin{cases}
\lambda_v \cdot f\big(\pa v\big) & \text{ if } v\in V^\circ,
   \\
0 & \text{ if } v=\koo.
   \end{cases}
   \end{align*}
By a {\em weighted shift on $\tcal$ with weights} $\lambdab=\{\lambda_v\}_{v \in V^{\circ}} \subseteq \C$, we mean the operator $\slam$ in $\ell^2(V)$ defined as follows
   \begin{align*}
   \begin{aligned}
\dz {\slam} & = \big\{f \in \ell^2(V) \colon \varLambda_\tcal^\lambdab f \in \ell^2(V)\big\},
   \\
\slam f & = \varLambda_\tcal^\lambdab f, \quad f \in \dz{\slam}.
\end{aligned}
\end{align*}
We deal with weighted shifts with positive weights throughout the paper. Since any weighted shift with non-zero weights is unitarily equivalent to a weighted shift with positive weights (see \cite{j-j-s-2012-mams}), the assumption of positivity of weights is not restrictive.

To avoid further repetitions we gather below the most basic assumptions: 
\begin{align} \label{stand2}\tag{$\star$}
   \begin{minipage}{70ex}
$\tcal=(V,E)$ is a countably infinite rooted and leafless directed tree,  and $\lambdab=\{\lambda_v\}_{v \in V^\circ}\subseteq  (0,\infty)$,
   \end{minipage}
   \end{align}
  and
\begin{align} \label{stand1}\tag{$\dag$}
   \begin{minipage}{70ex}
$\tcal=(V,E)$ is a countably infinite rooted directed tree, and \\$\lambdab=\{\lambda_v\}_{v \in V^\circ} \subseteq  (0,\infty)$.
   \end{minipage}
   \end{align}
In our previous work we used a notion of a multiplier algebra induced by a weighted shift, which is defined via related multiplication operators. These are given as follows. 
Given $u\in V$ and $v\in\dess(u)$ we set 
\begin{align*}
\lambda_{u|v}=\begin{cases}
 1 & \text{ if } u=v,\\
 \prod_{n=0}^{k-1} \lambda_{\pan{n}{v}} & \text{ if } \pan{k}{v}=u \text{ for } k \in \N.\end{cases}    
\end{align*} 
Assume \eqref{stand2}. Let $\hat\varphi \colon \N_0 \to \C$. Define the mapping $\varGamma_{\hat\varphi}^\lambdab\colon \C^V\to \C^V$ by the formula
\begin{align}\label{multiKL1}
\big(\varGamma_{\hat\varphi}^\lambdab f\big)(v) =\sum_{k=0}^{|v|} \lambda_{\paa^k(v)|v} \, \hat\varphi (k) f\big(\paa^k(v)\big),\quad v\in V.
\end{align}
The {\em multiplication operator} $M_{\hat\varphi}^{\lambdab}\colon \ell^2(V)\supseteq \dzD{M_{\hat\varphi}^{\lambdab}}\to \ell^2(V)$  is given by
   \begin{align*} 
   \begin{aligned}
\dzD{M_{\hat\varphi}^{\lambdab}} & = \big\{f \in \ell^2(V) \colon \varGamma_{\hat\varphi}^\lambdab f \in \ell^2(V)\big\},
   \\
M_{\hat\varphi}^{\lambdab} f & = \varGamma_{\hat\varphi}^\lambdab f, \quad f \in \dzD{M_{\hat\varphi}^{\lambdab}}.
\end{aligned}
\end{align*}
It is easily seen that for $u\in V$ such that $e_u\in\dzD{\mphi^\lambdab}$ we have
\begin{align}\label{dziobak}
\big(\mphi^\lambdab e_u\big)(v)=\left\{\begin{array}{cl}
\lambda_{u|v} \hat\varphi\big(|v|-|u|\big) & \text{ if } v\in\dess(u),\\
0 & \text{ otherwise}.
\end{array}\right.
\end{align}

Let $\multi$ denote the {\em multiplier algebra} induced by $\slam$, i.e., the commutative Banach algebra consisting of all $\hat\varphi\colon \N_0\to \C$ such that $\dzD{\mphi^\lambdab}= \ell^2(V)$ with the norm
\begin{align*}
\|\hat\varphi\|:=\big\|\mphi^\lambdab\big\|,\quad \hat\varphi\in \multi.
\end{align*}
Every $\hphi \in \multi$ is called a {\em multiplier} for $\slam$.
For more information on $\multi$ we refer the reader to \cite{b-d-p-2015}.

\section{Generalized multipliers}
In this section, we are going to state the definition of generalized multipliers for left-invertible analytic operator.   
Let us recall that
 $T\in\bsb(\hh)$ is \textit{left-invertible} if there is an operator $A\in\bsb(\hh)$ such that $AT=I$, and \textit{analytic} if $\bigcap_{n=0}^\infty T^n(\hh)=\{0\}$.
It is known that $T$ is left-invertible if and only if $T$ is bounded from below, i.e. there exists constant $\alpha>0$ such that $T^*T\geqslant \alpha I$. Let $T\in\bsb(\hh)$ be left-invertible. Define $\ee_T=\nul{T^*}$ and an operator $T':=T(T^*T)^{-1}$. If it is clear which operator stands for $T$, we will simply write $\ee$ instead of $\ee_T$. The operator $T'$ is \textit{called the operator Cauchy dual of} $T$. 
In what follows we denote the operator $T'^*$ by $L$.
It is easily seen that \begin{equation}\label{leftinverse}
 LT=I, \quad \nul{L}=\ee_T, \quad P_{\ee_T} = I - TL  \quad \text{and} \quad  L P_{\ee_T} =0.
 \end{equation}
 Denote by $\aa(\ee)$ the set of all $\ee$-valued analytic functions on $\Dbb_r$, where $r:=\frac{1}{r(L)}$. By definition $\ff\in\aa(\ee)$ if there is a sequence $\{\hat\ff(n)\}_{n=0}^\infty\subset \ee$ such that $\ff(z)=\sum_{n=0}^\infty \hat\ff(n)z^n$ for every $z\in\Dbb_r$.
 
 Now, we are ready to recall the Shimorin's model for a left-invertible analytic operator $T \in \bsb(\hh)$ (see \cite{shim}).
 Shimorin showed that the function 
\begin{equation}\label{ufor}(Uf)(z)=\sum_{n=0}^\infty(P_\ee L^nf)z^n,  \quad z\in \Dbb_r.\end{equation}
is well-defined and analytic in $\Dbb_r$ for every $f\in\hh$.
Moreover, transformation $U\colon\hh\to\aa(\ee)$ is injective, since $T$ is analytic. This enables him to define a scalar product on $U(\hh)$, such that $U$ is an isometry. By $\calh$ we denote the set $U(\hh)$, which is a Hilbert space of $\ee$-valued analytic functions.
By \cite[pp.154]{shim} the operator $T$ is unitary equivalent to the operator $\unitT$ of multiplication by $z$ on $\calh$ and $L$ is unitary equivalent to the operator $\unitL \in \bsb(\calh)$ given by the formula
\begin{align*}
    \unitL(\ff)(z) = \frac{\ff(z) - \ff(0)}{z}.
\end{align*}
In particular,
\begin{align} \label{unity}
    LU^* = U^* \unitL, \quad TU^* = U^* \unitT,
\end{align}
which can be expressed on the diagram as:

\begin{equation*}
\begin{split}
\xymatrix{
\calh \ar[d]_{U^*} \ar@<0.4ex>@{->}[r]^{\unitT} \ar@<-0.4ex>@{<-}[r]_{\unitL} & \calh \ar[d]^{{U}^*} \\
\hh \ar@<0.4ex>@{<-}[r]^{L} \ar@<-0.4ex>@{->}[r]_T & \hh,}
\end{split} \end{equation*}

Let us note that every $\ff\in \calh$ can be represented as follows
\begin{equation} \label{rozklad}
\ff=\sum_{n=0}^\infty \hat\ff(n) z^n,
\end{equation}
where $\hat\ff(n):=P_\ee L^nU^*\ff$ for $n\in\N_0$.
It is worth to notice the following observation: 
\begin{align}\label{tomek}
\begin{minipage}{60ex}
If $\{\ff_n \}_{n=0}^\infty \subset \calh$ is such that $\ff_n$ converges to some $\ff \in \calh$, then $\hat{\ff}_n(m) \toe \hat{\ff}(m)$ for every $m \in \N_0$.
\end{minipage}
\end{align}
Let $k_{\calh}$ be the reproducing kernel of $\calh$, i.e., $k_{\calh} \colon \Dbb_r \times \Dbb_r \to \bsb(\ee)$ is such that 
\begin{equation*}
    k_\calh(z,\lambda)=P_\ee(I-zL)^{-1}(I-\bar{\lambda}L^*)^{-1}|_\ee.
\end{equation*}
In particular for every $ e \in \ee$ and $\lambda \in \Dbb_r$ the function $k_\calh$ satisfies the following conditions \cite[pp.154--155]{shim}:
\begin{align} \label{wlasnosci}
k_{\calh}(\cdot, \lambda) e \in \calh, \quad \text{and }
\langle \ff(\lambda), e \rangle_{\ee} = \big\langle \ff, k_{\calh} (\cdot, \lambda)e \big\rangle_{\calh}, \text{ for every } \ff \in \calh.
\end{align}
In our previous work \cite{b-d-p-2015} we used a notion of a multiplier algebra induced by a weighted shift, which is defined via related multiplication operators. We are going to extend this notion to the  case of $\ee_T$-valued analytic functions.   
Let us consider the Cauchy-type multiplication $*\colon\bsb(\ee)^{\N_0}\times\ee^{\N_0}\to\ee^{\N_0}$ given by 
\begin{equation}
\label{multi0}
(\hphi*\hat{f})(n) = \sum_{k=0}^n \hphi(k)\hat{f}(n-k),  \quad \hphi\in\bsb(\ee)^{\N_0},\, \hat{f}\in\ee^{\N_0}.
\end{equation} 
The {\em multiplication operator} $\mphi\colon \calh\supseteq \dz{\mphi}\to \calh$ is defined as follows. The domain is given by
   \begin{align*}
   \begin{aligned}
\dz{\mphi} & = \big\{\ff \in \calh \colon \text{ there is } \gg\in\calh \text{ such that }\hphi*\hat{\ff}=\hat{\gg}\big\}.
\end{aligned}
\end{align*}
 By the uniqueness of power series for every $\ff\in\dz{\mphi}$ there exists exactly one $\gg\in\calh$ satisfying equality $\hphi*\hat{\ff}=\hat{\gg}$. In this situation, we set 
    \begin{align}
   \begin{aligned}\label{multi3}
\mphi \ff & = \gg, \quad \ff \in \dz{\mphi}.
\end{aligned}
\end{align}
We call $\hat\varphi \colon \N_0\to \bsb(\ee)$ the {\em symbol} of $\mphi$. 
Below, we prove that any multiplication operator $\mphi$ is closed. 
\begin{lem}\label{multipod}
Let $\hphi \colon \N_0 \to \bsb(\ee)$. Then the following conditions hold: 
\begin{enumerate}
\item[(i)] for every $e\in\ee$ and $n\in\N_0$, $(\hphi*\widehat{Ue})(n)=\hphi(n)e$,
\item[(ii)] for every $\ff\in\dz{\mphi}$ and $n\in\N_0$, we have the equality
\begin{align}\label{multi1}
\widehat{\mphi \ff}(n) = \sum_{k=0}^n \hphi(k) \hat \ff(n-k)= \sum_{k=0}^n \hphi(k) P_{\ee} L^{n-k}U^* \ff,
\end{align}
 \item[(iii)] for every every $\ff\in\dz{\mphi}$ and $z\in\Dbb_r$, the series $\sum\limits_{n=0}^\infty \big(\sum\limits_{k=0}^n \hphi(k) P_{\ee} L^{n-k}U^* \ff\big)z^n$ is convergent  and  $$\big(\mphi\ff\big)(z)=\sum_{n=0}^\infty \Big(\sum_{k=0}^n \hphi(k) P_{\ee} L^{n-k}U^* \ff\Big)z^n.$$
 \item[(iv)]the operator $\mphi$ is closed.
\end{enumerate}
\end{lem}
\begin{proof} (i) Fix $e\in\ee$ and let $\ff=Ue$. Then by \eqref{multi0}, \eqref{rozklad}, and equality $\nul{L} = \ee$ we get 
\begin{align*}
\big(\hphi*\widehat{Ue}\big)(n) =\sum_{k=0}^n \hphi(k) P_{\ee} L^{n-k}e=\hphi(n)e,  \quad n\in\N_0.
\end{align*}

(ii) It follows from \eqref{multi3}, \eqref{multi0}, and \eqref{rozklad}.

(iii) See \cite[Equality (2.3)]{shim}.

(iv)
Let $\{\ff_n\}_{n=0}^\infty \subset \dz{\mphi}$ be such that $\ff_n \to \ff$ in $\calh$  and $\mphi \ff_n \to \gg$ for some $\ff, \gg \in \calh$. Then $U^* \ff_n \to U^* \ff$. By \eqref{multi1} and continuity of all considered operators, for any $m\in\N_0$, we obtain
\begin{equation}\label{mphif}\widehat{M_{\hat{\varphi}} \ff_n}(m)=\sum_{k=0}^m\hphi(k)P_\ee L^{m-k}U^*\ff_n\to \sum_{k=0}^m\hphi(k)P_\ee L^{m-k}U^*\ff. \end{equation}
From the convergence of $\{\mphi \ff_n\}_{n=0}^\infty$ we deduce that $U^*\mphi \ff_n \to U^*\gg$. Thus 
$$\widehat{M_{\hat{\varphi}} \ff_n}(m)= P_\ee L^m U^*\mphi \ff_n \to P_\ee L^m U^*\gg=\hat{\gg}(m).$$
This and \eqref{mphif} imply that $\hat\gg(m)=\sum_{k=0}^m\hphi(k)P_\ee L^{m-k}U^*\ff$ for every $m\in\N_0$. Hence $\hphi*\hat{\ff}=\hat{\gg}$. Consequently, $\ff\in\dz{\mphi}$ and $\gg=\mphi\ff$, which completes the proof.
\end{proof}
If $\dz{\mphi}=\calh$, then $\mphi\in\bsb(\calh)$ since $\mphi$ is closed. In this situation, we call $\hphi$ a \textit{generalized multiplier} of $T$ and $\mphi$ a \textit{generalized multiplication operator by} $\hphi$. 
By $\calm(T)$ we denote the set of all generalized multipliers of the operator $T$. Let us notice that $\calm(T)$ is a linear subspace of $\bsb(\ee)^{\N_0}$ and the function $\| \cdot \| \colon \calm(T) \to [0,\infty)$ given by the formula 
$$\|\hat\varphi\|:=\|\mphi\|,\quad \hat\varphi\in \calm(T)$$
is a norm on $\calm(T)$, which can be deduced from Lemma \ref{multipod}(i) and the linearity of transformation $\calm(T)\ni\hphi\to \mphi\in\bsb(\calh)$.
By $\multiskal(T)$ we denote the linear subspace of $\calm(T)$ consisting of all generalized multipliers whose all coefficients are scalar multiples of the identity operator.
For a given operator $A \in \bsb(\hh)$ let $\hphi_{A} \colon \N_0 \to \bsb(\ee)$ be a sequence defined by the formula
\begin{align}\label{jani}
\hphi_{A}(m) &= P_{\ee} L^m A|_{\ee}, \quad m \in \N_0.
\end{align} 
The space $\calm(T)$ turns out to have a natural Banach algebra structure. It suffices to endow it with the Cauchy-type multiplication $*\colon \bsb(\ee)^{\N_0}\times\bsb(\ee)^{\N_0}\to \bsb(\ee)^{\N_0}$ given by
\begin{align}\label{cauchymult}
\big(\hat\varphi*\hat\psi\big) (k) = \sum_{j=0}^{k} \hat\varphi(j)\hat\psi(k-j),\quad \hat\varphi, \hat\psi\in \bsb(\ee)^{\N_0}.
\end{align}
\begin{thm}\label{cauchym}
Let $T\in\bsb(\hh)$ be left-invertible and analytic. Then following assertions are satisfied$:$
\begin{enumerate}
\item[(i)] For every $n\in\N_0$, the sequence $\chi_{\{n\}}\I_{\ee}$ is a generalized multiplier and $\unitT^n=M_{\chi_{\{n\}}\I_{\ee}}$.
\item[(ii)] If $\hphi \in \calm(T)$, then $\mphi$ commutes with $\unitT$.
\item[(iii)] For all $\hat\varphi, \hat\psi\in \calm(T)$, the function $\hat\varphi*\hat\psi$ belongs to $\calm(T)$ and
\begin{align*}
\mphi M_{\hat\psi}= M_{\hat \varphi*\hat\psi}.
\end{align*}
\item[(iv)] The spaces $\calm(T)$, $\multiskal(T)$ endowed with 
the Cauchy-type multiplication, are a Banach algebras with a unit $\chi_{\{0\}}\I_\ee$. In addition, $\multiskal(T)$ is commutative.
\end{enumerate}
\end{thm}
\begin{proof}
(i) Let $\ff \in \calh$. Define $\gg = \unitT^n \ff$ and $\hphi=\chi_{\{n\}}\I_{\ee}$. Then $ \hphi  * \hat{\ff} = \hat{\gg}$. Thus $\dz{\mphi} =
\calh$. Hence, by \eqref{multi3}, the equality $\mphi\ff=\unitT^n \ff$ holds for every $ \ff \in \calh$.

(ii) For every $\ff \in \calh$ and $z \in \Dbb_r$, by Lemma \ref{multipod}(iii), \eqref{unity} and \eqref{leftinverse} the following equalities hold
\begin{align*}
   \big(\mphi \unitT \ff\big)(z) &=  \sum_{n=0}^\infty \Big(\sum_{k=0}^n \hphi(k)P_{\ee}L^{n-k} U^*\unitT \ff \Big)z^n 
    = \sum_{n=1}^\infty \Big(\sum_{k=0}^{n-1} \hphi(k)P_{\ee}L^{n-1-k} U^* \ff \Big)z^n \\
    &= z \sum_{n=0}^\infty \Big(\sum_{k=0}^{n} \hphi(k)P_{\ee}L^{n-k} U^* \ff \Big)z^{n} 
    = z \mphi \ff(z)  = \big(\unitT \mphi \ff\big)(z).
\end{align*}
(iii) First, notice that by \eqref{multi3}, \eqref{multi0}, changing the order of summation and \eqref{cauchymult} we have the following equalities
\begin{align*}
    (\mphi\mpsi \ff)\widehat{ \phantom{a} }\  (n)&= 
     \sum_{k=0}^n \hphi(k)\widehat{\mpsi \ff}(n-k) 
    = \sum_{k=0}^n \hphi(k)\sum_{j=0}^{n-k} \hat\psi(j) \hat{\ff}(n-k-j)  \\
    &= \sum_{l=0}^n\sum_{j=0}^{l} \hphi(j) \hat\psi(l-j)\hat{\ff}(n-l)
    = \sum_{l=0}^n(\hphi * \hat\psi)(l)\hat{\ff}(n-l), \quad \ff \in \calh, \, n\in\N_0.
\end{align*}
Hence $ (\hphi * \hpsi)* \hat{\ff} = (\mphi\mpsi \ff)\widehat{ \phantom{a} }$ for every $\ff\in\calh$. Thus, $\dz{M_{\hat\varphi *\hat\psi}}=\calh$ and $M_{\hat\varphi *\hat\psi}= \mphi \mpsi$.

(iv) Let $\{ \hphin\}_{n=0}^\infty  \subset \calm(T)$ be a Cauchy sequence. Then there exists an operator $\cala \in \bsb(\calh)$ such that $\lim_{n\to \infty} \mphin = \cala$. Thus \eqref{tomek} implies that 
\begin{align}\label{granica}
 \widehat{\mphin \ff}(m) \toe \widehat{\cala \ff}(m), \quad \ff\in\calh,\, m\in\N_0. 
\end{align} Let us define a sequence $\hphi:=\hphi_{U^*\cala U}$ as in \eqref{jani}. First, we are going to prove that $\hphin(m)\tosot\hphi(m)$ for every $m\in\N_0$. Fix $e\in\ee$ and let $\gg=Ue$.
Then by Lemma \ref{multipod} (i), \eqref{granica}, \eqref{rozklad} and \eqref{jani} we obtain
\begin{align*} 
   \hphin(m)e =  \widehat{\mphin\gg}(m)\toe \widehat{\cala\gg}(m) = P_{\ee}L^mU^*\cala Ue=\hphi(m)e, \quad m \in \N_0.
\end{align*}
 This proves that $\hphin(m)\tosot\hphi(m)$ for every $m\in\N_0$. Moreover, if $\{ \hphin\}_{n=0}^\infty  \subset \multiskal(T)$ then $\hphi(m)$ is a scalar multiple of the identity operator for every $m \in \N_0$. By  \eqref{granica}, \eqref{multi1}, the SOT-convergence of $\{\hphi_n\}_{n=0}^\infty$, \eqref{rozklad} and \eqref{multi0},
we get the equality $\widehat{\cala \ff}(m)=(\hphi*\hat{\ff})(m)$ for every $m\in\N_0$ and $\ff \in \calh$. Hence $\dz{\mphi} = \calh$, $\hphi \in \calm(T)$ (resp. $\hphi\in\multiskal(T)$), and $\cala  = \mphi$. The other conditions are satisfied from the definition and (iii).
\end{proof}
It is worth to notice that every sequence of complex numbers with a finite support is a multiplier for a weighted shift on rooted directed tree. However, the sequence $\hphi \colon \N_0 \to \bsb(\ee)$ with finite support does not have to be a generalized multiplier. Moreover, the operation of permutation of even two coefficients of generalized multiplier is not closed in this space. This shows that the structure of generalized multipliers is very delicate.

\begin{figure}[ht]
\begin{tikzpicture}[scale=0.8, transform shape,edge from parent/.style={draw,to}]
\tikzstyle{every node} = [circle,fill=gray!30]
\node (e10)[font=\footnotesize, inner sep = 1pt] at (0,0) {$(0,0)$};

\node (e11)[font=\footnotesize, inner sep = 1pt] at (3,1) {$(1,1)$};
\node (e12)[font=\footnotesize, inner sep = 1pt] at (6,1) {$(1,2)$};
\node (e13)[font=\footnotesize, inner sep = 1pt] at (9,1) {$(1,3)$};
\node[fill = none] (e1n) at(12,1) {};

\node (f11)[font=\footnotesize, inner sep = 1pt] at (3,-1) {$(2,1)$};
\node (f12)[font=\footnotesize, inner sep = 1pt] at (6,-1) {$(2,2)$};
\node (f13)[font=\footnotesize, inner sep = 1pt] at (9,-1) {$(2,3)$};
\node[fill = none] (f1n) at(12,-1) {};

\draw[->=stealth] (e10) --(e11) node[pos=0.5,above = 0pt,fill=none] {$1$};
\draw[->] (e11) --(e12) node[pos=0.5,above = 0pt,fill=none] {$1$};
\draw[->] (e12) --(e13) node[pos=0.5,above = 0pt,fill=none] {$1$};
\draw[dashed, ->] (e13)--(e1n);
\draw[->] (e10) --(f11) node[pos=0.5,below = 0pt,fill=none] {$\alpha$};
\draw[->] (f11) --(f12) node[pos=0.5,below = 0pt,fill=none] {$\alpha$};
\draw[->] (f12) --(f13) node[pos=0.5,below = 0pt,fill=none] {$\alpha$};
\draw[dashed, ->] (f13)--(f1n);
\end{tikzpicture}
\caption{\label{truffaz-fig}}
\end{figure}
\begin{exa}\label{1alfa}
Let $\tcal_{2} = (V_{2},{E_{2}})$ be the directed tree with one branching vertex, given by (see Figure \ref{truffaz-fig})\allowdisplaybreaks
\begin{align*}
    V_{2} &= \big\{(0,0) \big\} \cup \big\{ (i,j) \colon  i\in\{1,2\},\ j \in \N \big\}, \\
    E_{2} &= \Big\{ \big((0,0),(i,1)\big) \colon i\in\{1,2\} \Big\}\cup \Big\{ \big((i,j),(i,j+1)\big) \colon i\in\{1,2\},\ j \in \N \Big\}.
\end{align*}
Let $\alpha \in (0,1)$. Let $\slam$ be a weighted shift on $\tcal_{2}$ with weights $\lambdab = \{ \lambda_v\}_{v \in V_{2}^\circ}$ defined as follows
\begin{align*}
\lambda_{(i,j)} 
= \left\{ 
\begin{array}{cl} 
1 & \text{ for } i=1 \text{ and } j\in\N, \\ 
\alpha & \text{ for } i=2 \text{ and } j\in\N.
\end{array} 
\right.
\end{align*}
Then $\slam \in \bsb(\ell^2(V_2))$ is left-invertible and analyic, $\big\{ e_{00}, \alpha e_{11} - e_{21}\big\}$ is a basis of $\nul{\slam^*}$, and 
$$ \slam^k (e_{00}) = e_{1k} + \alpha^k e_{2k}, \quad \slam^k(\alpha e_{11} - e_{21}) = \alpha e_{1, k+1} - \alpha^k e_{2,k+1} \text{ for } k \in \N.$$
Moreover,
$$ \slam^* \slam e_{00} = (\alpha^2+1)e_{00}, \quad \slam^* \slam e_{1k} = e_{1k}, \quad \slam^*\slam e_{2k} = \alpha^2 e_{2k}\text{ for }  k \in \N.$$
Thus for $n \geq 1$ and $ f \in \ell^2(V_2)$ we have
$$ L^n f = \bigg(\frac{f(1,n)}{\alpha^2 +1} + \frac{f(2,n)}{\alpha^{n-2}(\alpha^2+1)}\bigg)e_{00} + \sum_{k=1}^\infty \bigg(f(1,k+n) e_{1k}  + \frac{f(2,k+n)}{\alpha^n}e_{2k}\bigg).$$
As a consequence
\begin{align} \label{rzut}
\begin{aligned}
  P_{\ee} L^n f &= \frac{1}{\alpha^2+1}\bigg(f(1,n) + \alpha^{2-n}f(2,n)\bigg)e_{00}\\ &+\frac{1}{\alpha^2+1} \bigg(\alpha f(1,1+n) - \alpha^{-n}f(2,1+n)\bigg)\big( \alpha e_{11} - e_{21}\big),\, n\geqslant 1.   
  \end{aligned}
\end{align}
Let $\hphi \colon \N_0 \to \bsb(\ee)$ be defined as $ \hphi(n) =\left\{ \begin{array}{cl}  A_0 &\text{ if } n =0 \\ 0 &\text{ if } n >0 \end{array}\right.$, where 
\begin{align*}
    A_0(e_{00})  = a e_{00} + c \big( \alpha e_{11} - e_{21}\big), \quad A_0( \alpha e_{11} - e_{21})  = b e_{00} + d \big( \alpha e_{11} - e_{21}\big), \, a, \, b, \, c, \, d \in \C.
\end{align*}
We are going to show that $\hphi\in\calm(\slam)$ if and only if $b=c=0$ and $a=d$. If $b=c=0$ and $a=d$ then by Theorem \ref{cauchym}(i) we get $\hphi \in \calm(\slam)$. 
Let us notice that $ \hphi \in \calm(\slam)$ if and only if for every $f \in \ell^2(V)$ there exists $g \in \ell^2(V)$ such that $A_0 P_{\ee} L^n f  =\widehat{\mphi  Uf}(n)  = P_{\ee} L^n g$ for every $n \in \N_0$. By \eqref{rzut}
\begin{multline*}
    A_0 P_{\ee}L^n f = \frac{1}{\alpha^2+1}\bigg(af(1,n) + a\alpha^{2-n}f(2,n)+b\alpha f(1,1+n) - b\alpha^{-n}f(2,1+n)\bigg)e_{00} \\ \quad + \frac{1}{\alpha^2+1}\bigg(cf(1,n) + c\alpha^{2-n}f(2,n)+d\alpha f(1,1+n) - d\alpha^{-n}f(2,1+n)\bigg)\big( \alpha e_{11} - e_{21}\big), \quad n\in\N
\end{multline*}
and 
\begin{align}\label{g1n}
\begin{aligned}
    (1+\alpha^2)g(1,n) &=\alpha cf(1,n-1)+(a+d\alpha^2)f(1,n) +b\alpha f(1,1+n) 
    \\ &+ c\alpha^{4-n}f(2,n-1)+ (a-d)\alpha^{2-n}f(2,n) -b\alpha^{-n}f(2,1+n), \quad n\geqslant2.
    \end{aligned}
\end{align}
Now, let $f\in\ell^2(V)$ be defined as follows
$f(u)= \left\{ \begin{array}{cl}  \alpha^{|u|} &\text{ if } u=(2,|u|) \text{ and } |u| \text{ is divisible by } 3, \\ 0 &\text{ otherwise. }\end{array}\right.$
By this and \eqref{g1n} we obtain
\begin{align*}
     \sum_{n=1}^\infty |g(1,3n)|^2 = \sum_{n=1}^\infty \frac{\alpha^4|a-d|^2}{(\alpha^2 +1)^2}.
\end{align*}
Since $ g \in \ell^2(V)$ we obtain that $a=d$. In a similar manner we prove that $b=c=0$.

Now, let $\hpsi \colon \N_0 \to \bsb(\ee)$ be defined as $ \hpsi(k) =\left\{ \begin{array}{cl}  A_k &\text{ if } k < 2 \\ 0 &\text{ if } k \geqslant 2 \end{array}\right.$, where 
\begin{align*}
    A_k =  \left( \begin{array}{cc} a_k & c_k \\ b_k & d_k \end{array} \right),
\end{align*}
with respect to the basis $\{e_{00},\alpha e_{11} - e_{21}\}$. Then, one can show similarily that $\hpsi$ is a generalized multiplier if and only if 
\begin{align*}
    A_0 =  \left( \begin{array}{cc} a_0 & 0 \\ \frac{d_1-a_1}{\alpha} & d_0 \end{array} \right) \text{ and } A_1 =  \left( \begin{array}{cc} a_1 & (a_0-d_0)\alpha \\ 0 & d_1 \end{array} \right).
\end{align*}
\end{exa}
The next theorem provides generalized analytic structure for the commutant of left-invertible, analytic operator.
\begin{thm} \label{komutant}
Let $T \in \bsb(\hh)$ be left-invertible and analytic. Assume that $A \in \bsb(\hh)$ commutes with $T$. Then $\hphi_A\in\calm(T)$ and $A = U^* M_{\hphi_A}U$.
\end{thm}
\begin{proof}
Let $\cala=UAU^*$.
By \eqref{leftinverse}, unitary equivalence, and commutation we get the equality
\begin{align}\label{nowe}
\begin{aligned} 
P_{\ee} AP_{\ee}U^* \ff &= P_{\ee} A (I - TL)U^* \ff = P_{\ee} AU^* \ff  - P_{\ee} ATLU^* \ff\\
&=P_{\ee} U^* \cala\ff  - P_{\ee} TALU^* \ff = P_{\ee} U^* \cala \ff, \text{ for every } \ff\in\calh.
\end{aligned}
\end{align}
Now, let $\hphi := \hphi_A$. We will prove by the induction that $(\hphi*\hat{\ff})(n) = \widehat{\cala \ff}(n)$ for every $n \in \N_0$ and $\ff \in \calh$. By \eqref{multi0}, \eqref{rozklad}, \eqref{jani} and \eqref{nowe} we obtain 
\begin{align*}
    (\hphi*\hat{\ff})(0)=\hphi(0)P_\ee U^*\ff=P_{\ee}A P_{\ee} U^* \ff=
    P_\ee U^*\cala \ff=\widehat{\cala \ff}(0), \quad \ff \in \calh.
\end{align*} 
Now let $n\in\N$. Then, by \eqref{multi0}, \eqref{rozklad}, \eqref{jani}, \eqref{leftinverse}, \eqref{unity}, inductive hypothesis, and commutation we obtain the following equalities
\begin{align*}
 (\hphi*\hat{\ff})(n)&=  \sum_{k=0}^n \hphi(k) P_{\ee} L^{n-k}U^* \ff
 = \sum_{k=0}^{n-1} \hphi(k) P_{\ee} L^{n-k}U^* \ff+ P_{\ee} L^n A P_{\ee} U^* \ff\\
 &= \sum_{k=0}^{n-1} \hphi(k) P_{\ee} L^{n-1-k}U^*\unitL \ff+P_{\ee} L^n AU^* \ff -P_{\ee} L^n ATL U^* \ff\\
 &= (\hphi*\widehat{\unitL \ff})(n-1)+P_{\ee} L^n U^* \cala \ff -P_{\ee} L^nT ALU^* \ff\\
 &= \widehat{\cala\unitL \ff}(n-1)+\widehat{\cala \ff}(n) -P_{\ee} L^{n-1} U^* \cala \unitL \ff\\
  &= \widehat{\cala\unitL \ff}(n-1)+\widehat{\cala \ff}(n)-\widehat{\cala \unitL\ff}(n-1)
  =\widehat{\cala \ff}(n).
\end{align*}
Hence, $\hphi \in \calm(T)$ and $\cala =\mphi$.
\end{proof}
By the above and Theorem \ref{cauchym}(ii) we deduce the following corollaries.
\begin{cor}
Let $T \in \bsb(\hh)$ be left-invertible and analytic. Then the commutant of $T$ is unitary equivalent to the algebra of all generalized multipliers of $T$. 
\end{cor}
\begin{cor} \label{wniosekkom}
Let $T \in \bsb(\hh)$ be left-invertible and analytic. Then the algebra of all generalized multipliers of $T$ is closed in \textrm{SOT} and \textrm{WOT} topology. 
\end{cor}
At the end of this section we state two general lemmas, which will be used later in the proof of Proposition \ref{skalarne} in the context of weighted shifts on directed trees.
\begin{lem}\label{kosiarka} Let $T \in \bsb(\hh)$ be left-invertible and analytic.
If $\{ \hphin\}_{n=0}^\infty\cup\{\hphi\} \subset \calm(T)$ and $\mphin \tosot \mphi$, then $\hphin(m) \tosot \hphi(m)$ for every $m\in\N_0$.
\end{lem}
\begin{proof}
Fix $e\in\ee$ and let $\ff=Ue$. 
 By our assumptions $\mphin \ff \toh \mphi \ff$. Then Lemma 1(i) and \eqref{tomek} imply
 \begin{align*} 
  \hphin(m)e = \widehat{\mphin \ff}(m) \toe \widehat{\mphi \ff}(m) = \hphi(m)e , \quad m\in\N_0.
 \end{align*}
 This proves the Lemma, since $e\in\ee$ is arbitrary.
\end{proof}
\begin{lem}\label{mnoz}
Let $T\in\bsb(\hh)$ be left-invertible and analytic, $\hphi\in\multiskal(T)$ and let $\{a_n\}_{n=0}^\infty\subset\C$ be such that $\hphi(n)=a_nI_\ee$ for every $n\in\N_0$. Then 
\begin{equation}\label{phiconv}
 \sum_{n=0}^\infty a_nz^n \text{ is convergent for every } z\in\Dbb_r
 \end{equation}
 and
\begin{equation}\label{wzormnoz}
(\mphi\ff)(z)=\left(\sum_{n=0}^\infty a_nz^n\right)\ff(z) \text{ for every } z\in\Dbb_r \text{ and } \ff\in\calh.
\end{equation}
\end{lem}
\begin{proof}
Let $z\in\Dbb_r$ and let $e \in \ee$ be such that $\|e\|_{\hh}=1$. Then $(Ue)(z)=e$, since $e\in\ee$. As a consequence we get
 \begin{align*}
 \big\langle (\mphi Ue)(z),e\big\rangle_\ee=\Big\langle\sum_{n=0}^\infty a_nez^n,e\Big\rangle_\ee
 =\sum_{n=0}^\infty a_n\big\langle z^ne,e\big\rangle_\ee=\sum_{n=0}^\infty a_nz^n.
 \end{align*}
 In particular $\sum_{n=0}^\infty a_nz^n \text{ is convergent for every } z\in\Dbb_r$.
 
 Now, let $\ff\in \calh$. It is easily seen that $\is{\ff(z)}{e}=\sum_{n=0}^\infty\is{\hat\ff(n)}{e}z^n$ and hence the series $\sum_{n=0}^\infty\is{\hat\ff(n)}{e}z^n$ is convergent for every $z\in\Dbb_r$.  By this and \eqref{phiconv} we obtain
 \begin{align*}
     \bis{(\mphi \ff)(z)}{e}  &=\sum_{n=0}^\infty \sum_{k=0}^n a_k\is{ \hat{\ff}(n-k)}{e} z^n   
      \\ &= \Big( \sum_{n=0}^\infty a_n z^n \Big) \cdot \Big(\sum_{n=0}^\infty \is{ \hat{\ff}(n)}{e} z^n \Big)=\Big( \sum_{n=0}^\infty a_n z^n \Big)\is{\ff(z)}{e}.
 \end{align*}
 Thus \eqref{wzormnoz} holds, since $e\in\ee$ is arbitrary.
\end{proof}

\section{Weighted shifts on directed trees}
One of the most important examples of analytic operators are bounded weighted shift operators on directed trees with root (see \cite[Lemma 3.3]{c-t}). This class is a subclass of weighted composition operators (see \cite{b-j-j-s-ampa}). Our aim in this section is to show that generalized multipliers with scalar coefficients (defined in the previous section) and multipliers (defined in \cite{b-d-p-2015}) coincide in the case of left-invertible bounded weighted shift $\slam$ on rooted directed tree. In order to do this, we will prove that operators unitarily equivalent to polynomials of $\slam$ are SOT dense in the space $\{\mphi\colon \hphi\in\multiskal(\slam)\}$.
     
   Let us define the $k$-th generation of vertices as the set $V_k:= \{ v \in V \colon |v|=k\}$, for some $k \in \N_0$. Functions acting on $k$-th generation of vertices forms the set $$\ell^2(V_k)=\big\{f\in\ell^2(V) \colon f(u)=0 \text{ if } |u|\neq k\big\}, \quad k \in \N_0.$$ By $P_k$ we denote the orthogonal projection from the space $\ell^2(V)$ onto $\ell^2(V_k)$.
   Let us state a simple, but very useful lemma, which enables us to consider special type of $\ell^2(V)$ bases.
   \begin{lem} \label{separated}
    Assume \eqref{stand2}. Let $\slam\in\bsb(\ell^2(V))$ be left-invertible.  Then there exists an orthonormal basis $\{e'_j\}_{j\in J}$ of $\nul{\slam^*}$ such that for every $j \in J$ vector $e'_j$ belongs to the space $\ell^2(V_{k_j})$ for some $k_j \in \N_0$.
\end{lem}
\begin{proof}
It follows from the orthogonal decomposition of $\nul{\slam^*}$ (cf. \cite[equation (23)]{b-d-p-p-2017}).
\end{proof}
Any orthonormal basis satisfying condition from Lemma \ref{separated} will be called {\em separated}.

In \cite{b-d-p-p-2017}, we defined the function $f_w\colon V\to\C$ for $w \in \T$ and $f\colon V\to\C$ (the rotation of $f$ by the angle $w$) by the formula \begin{align*}
f_w(u) = w^{|u|} f(u),\quad u\in V.
\end{align*} This enabled us to show that polynomials of the operator are SOT-closed in the space of its multipliers. Before we state the definition of rotation in $\calh$, we need to define special diagonal operators.
 \begin{dfn}
 Assume \eqref{stand2}. Let $\slam\in\bsb(\ell^2(V))$ be left-invertible and let $\{e'_j\}_{j\in J}$ be a separarted basis of $\nul{\slam^*}$ such that $e'_j\in\ell^2(V_{k_j})$ for every $j\in J$, where $k_j\in\N_0$. For $ w \in \T$, we denote by $D_w\in\bsb\big(\nul{\slam^*}\big)$ a diagonal operator given by $$D_we'_j=w^{k_j}e'_j \text{ for every } j\in J.$$
 \end{dfn}
 Now, we are able to define the analogue of the rotation of function in $\calh$.
   \begin{dfn}\label{fw} Assume \eqref{stand2}. Let $\slam\in\bsb(\ell^2(V))$ be left-invertible and let $w \in \T$. 
If $\ff \in \calh$, then by $\ff_w$ we denote the analytic function $U(U^*\ff)_w\in\calh$. Let $\hphi\colon \N_0\to\bsb\big(\nul{\slam^*}\big)$. By $\hphi_w$ we denote the sequence $\{w^nD_w\hphi(n)D_{\overline{w}}\}_{n=0}^\infty$.
\end{dfn}

\begin{rem}
For $\ff \in \calh$ and $w \in \T$ one may define the function $\ff_w$ by the formula $\sum_{n=0}^\infty \hat{\ff}(n) w^n z^n$ (see \cite[Equality (50)]{shi}), which is an analytic function in $\aa\big(\nul{\slam^*}\big)$. However, in general such function need not to be in $\calh$ (the counterexample may be build for the weighted shift from Example \ref{1alfa}).
\end{rem}
The following proposition is an analogue of \cite[Lemma 3.1]{b-d-p-p-2017} for generalized multipliers.
\begin{prop}\label{cont}
Assume \eqref{stand2}. Let $\slam\in\bsb(\ell^2(V))$ be left-invertible and let $\{e'_j\}_{j\in J}$ be a separated basis of $\nul{\slam^*}$ such that $e'_j\in\ell^2(V_{k_j})$ for every $j\in J$, where $k_j\in\N_0$. Then the following conditions hold:
\begin{enumerate}
\item[(i)] $\widehat{\ff_w}(n)=w^nD_w\big(\hat{\ff}(n)\big)$, for every $\ff\in\calh$, $w\in\T$, and $n\in\N_0$.
\item[(ii)] For every $ \ff \in \calh$ and $w\in \T$, the vector $\ff_w$ belongs to $\calh$ and $\|\ff \| = \| \ff_w\|$.
\item[(iii)] For every $\ff \in \calh$, the mapping $\T \ni w \mapsto \ff_w \in \calh$ is continuous.
\item[(iv)] For every $\hphi \in \calm(\slam)$ and every $w\in\T$, the sequence $\hphi_w\in\calm(\slam)$ and 
$$ \mphiw \ff = (\mphi \ff_{\overline{w}})_w, \quad \ff \in \calh.$$ 
\item[(v)] For every $\hphi \in \calm(\slam)$, the mapping $\T \ni w \mapsto M_{\hphiw} \in \bsb(\calh)$ is SOT-continuous.
\end{enumerate}
\end{prop}
\begin{proof}
(i)  Let $f:=U^*\ff$. Then by Definition \ref{fw} and the facts that $L^{*n} e'_j\in\ell^2(V_{k_j+n})$ and $e'_j \in \nul{\slam^*}$ for every $n \in \N_0$, $j \in J$ we have
\begin{align*}
\bis{\widehat{\ff_w}(n)}{e'_j}&=\bis{P_{\nul{\slam^*}} L^n U^*(\ff_w)}{e'_j} = \langle  L^n f_w, e'_j \rangle = \langle f_w, L^{*n}e'_j \rangle \\ &=
\langle   f_w, P_{n+k_j}L^{*n} e'_j \rangle=\langle P_{n+k_j}  f_w, L^{*n} e'_j \rangle = \langle w^{n+k_j} P_{n+k_j}f, L^{*n} e'_j \rangle \\ &= 
 w^{n+k_j} \langle P_{\nul{\slam^*}}  L^{n}f,  e'_j \rangle = w^{n+k_j} \langle \hat{\ff}(n),  e'_j \rangle.
\end{align*}
As a consequence we obtain the equality
$$ \widehat{\ff_w}(n) = \sum_{j \in J} \langle \widehat{\ff_w}(n), e'_j \rangle e'_j=  \sum_{j \in J} w^{n+k_j}\langle \hat{\ff}(n), e'_j \rangle e'_j= w^n D_w \big(\hat{\ff}(n)\big).$$

(ii) and (iii) follow from \cite[Lemma 3.1]{b-d-p-p-2017}, since $U$ is unitary operator.

(iv) Let $\hphi \in \calm(\slam)$ and $w\in\T$. Then by \eqref{multi0}, Definition \ref{fw}, and (i)
\begin{align*}
 (\hphiw *\hat{\ff}) (n) &= \sum_{k=0}^n \hphiw(k) \hat{\ff}(n-k) = \sum_{k=0}^n w^k D_w \hphi(k)D_{\overline{w}} \hat{\ff}(n-k) \\
 &= w^n D_w \Big(\sum_{k=0}^n \hphi(k) \overline{w}^{n-k} D_{\overline{w}} \hat{\ff}(n-k)\Big) = w^n D_w \Big(\sum_{k=0}^n \hphi(k) \widehat{\ff_{\overline{w}}}(n-k)\Big)\\
 &= w^n D_w \big((\hphi *\widehat{\ff_{\overline{w}}}) (n)\big) = w^n D_w \big((\mphi \ff_{\overline{w}}) \widehat{\phantom{A}} (n)\big) \\ &=\big((\mphi\ff_{\overline{w}})_w\big)\widehat{\phantom{A}}(n), \quad \ff \in \calh, \, n \in \N_0.
\end{align*}
Thus $\hphiw\in\calm(\slam)$ and $ \mphiw \ff = (\mphi \ff_{\overline{w}})_w$.

(v) Let $\varepsilon>0$, $w\in\T$, and $\ff\in\calh$. Then by (iv), (iii), and (ii)
\begin{align*}
    \|\mphiw\ff-M_{\hphi_{w'}}\ff\|&=\|(\mphi \ff_{\overline{w}})_w-(\mphi \ff_{\overline{w'}})_{w'}\|\\
    & \leqslant \|(\mphi \ff_{\overline{w}})_w-(\mphi \ff_{\overline{w}})_{w'}\|+\|(\mphi \ff_{\overline{w}})_{w'}-(\mphi \ff_{\overline{w'}})_{w'}\|\\
    & <\varepsilon+\|\mphi \ff_{\overline{w}}-\mphi \ff_{\overline{w'}}\|\\
    & \leqslant \varepsilon+\|\mphi\|\cdot \| \ff_{\overline{w}}-\ff_{\overline{w'}}\|<2\varepsilon
\end{align*}
for $w'\in\T$ sufficiently close to $w$. 
\end{proof}

Proposition \ref{cont} enables us to define the integral 
$$ \int_{\T} q(w) \mphiw \D w\in \bsb(\calh),$$
for any continuous function $q \colon \T \to \C$. This can be done in a similar manner as in \cite[Section 3.1.2]{nik2} or \cite[Section 3]{b-d-p-p-2017}.

Let $\trig$ denotes the set of \textit{trigonometric polynomials} on $\T$. Given  $p\in\trig$ of degree $n\in\N_0$, i.e., $p(z)=\sum_{k=-n}^n p_k z^k$ with $\{p_k\}_{k=-n}^n\subseteq \C$, we define $\hat p \colon \N_0 \to \C$ to be the mapping such that \begin{align*}
\hat p(k) = \left\{ \begin{array}{cl} 
p_k & \text{if } k \leq n, \\
0 &  \text{if } k>n.
\end{array}
\right.
\end{align*}

In the next two propositions, we will prove that operators unitarily equivalent to polynomials of $\slam$ are dense in the space $\{\mphi\colon \hphi\in\multiskal(\slam)\}$ with respect to the SOT topology. In general, if $\hphi \in \calm(\slam)$ then the sequence $\hat p\, \hphi $ may not be a generalized multiplier. Hence, the same argumentation fails for $\calm(\slam)$ space.
\begin{prop} \label{calka}
Assume \eqref{stand2}. Let $\slam\in\bsb(\ell^2(V))$ be left-invertible. If $\hphi \in\multiskal(\slam)$ and $p\in\trig$, then $\hat{p}\, \hphi\in\multiskal(\slam)$ and

\begin{equation*}
  \int_{\T} p(\overline{w}) \mphiw \D w=M_{\hat{p} \, \hphi}.  
\end{equation*}
\end{prop}
\begin{proof} Let $\hphi \in \multiskal(\slam)$. To prove the claim it is sufficient to consider the function $p(w)=w^k$ for $k\in\Z$. By Theorem \ref{cauchym}(i) the sequence $\hat{p}\,\hphi$ belongs to $\multiskal(\slam)$.  
Let $\ff \in \calh$, $z \in \Dbb_r$ and $e \in \nul{\slam^*}$. Applying Lemma \ref{mnoz}, Definition \ref{fw}, and the Lebesgue dominated convergence theorem we get the equalities
\begin{align*}
\Bis{\Big( \int_{\T} \overline{w}^k \mphiw \D w \ff \Big)(z)}{e} &= \int_{\T} \Bis{\Big(  \overline{w}^k \mphiw \ff (z)}{e} \D w = \int_{\T} \overline{w}^k  \sum_{n=0}^\infty \hphiw(n) z^n \is{f(z)}{e} \D w   \\
&= \int_{\T} \overline{w}^k  \sum_{n=0}^\infty w^n \hphi(n) z^n \is{f(z)}{e} \D w \\
&= \is{f(z)}{e} \int_{\T}  \sum_{n=0}^\infty w^{n-k}  \hphi(n) z^n  \D w \\
&= \is{f(z)}{e}  \sum_{n=0}^\infty  \int_{\T} w^{n-k}  \hphi(n) z^n  \D w \\
&= \left\{ \begin{array}{cl} 
            \is{f(z)}{e} \hphi(k) z^k & \text{ if } k \geqslant 0 \\
            0 &                         \text{ if } k<0
            \end{array} \right. \\
&= \is{(M_{\hat p \, \hat\varphi} \ff)(z)}{e},
\end{align*}
which completes the proof.
\end{proof}
\begin{prop}\label{cesaro}
Assume \eqref{stand2}. Let $\slam\in\bsb(\ell^2(V))$ be left-invertible, $\hphi \in \multiskal(\slam)$ and let 
$$ p_n(w) = \sum_{k=-n}^n \big(1 - \frac{|k|}{n+1}\big) w^k, \quad  n \in \N.$$
Then 
\begin{enumerate}
\item[(i)] $\hat{p}_n \hphi \in \multiskal(\slam)$,
\item[(ii)] for every $n \in \N$, $\|M_{\hat{p}_n  \hphi}\| \leqslant \|\mphi\|$,
\item[(iii)] $M_{\hat{p}_n  \hphi} \to \mphi $ in the strong operator topology.
\end{enumerate}
\end{prop}
\begin{proof}
(i) It follows from Theorem \ref{cauchym}.

(ii) See the proof of \cite[Lemma 3.3(ii)]{b-d-p-p-2017} and use Propositions \ref{calka} and \ref{cont}.

(iii) Let $\unitT=U\slam U^*$. Applying equality $\ell^2(V) = \bigvee_{n=0}^\infty \slam^n \nul{\slam^*}$ (see \cite[Lemma 6.4]{b-d-p-p-2017}) we obtain that $\calh = \bigvee_{n=0}^\infty \unitT^n U \nul{\slam^*}$, where $\bigvee_{n=0}^\infty X_n$ stands for the smallest closed linear subspace of $\calh$ such that $X_k \subset \bigvee_{n=0}^\infty X_n$ for every $k \in \N_0$. Now, by 
(ii) and the fact that $\unitT$ commutes with $\mpsi$ for every $\hpsi\in\multiskal(\slam)$, it is sufficient to show that
$$ M_{\hat{p}_n \hphi} U e \to \mphi U e, \text{ for every } e \in \nul{\slam^*}.$$  
Let $\{e'_j\}_{j \in J}$ be a separated basis for $\nul{\slam^*}$ and let $e'_j \in \ell^2(V_{k_j})$ for some $k_j \in \N_0$. We will show that 
$$ M_{\hat{p}_n \hphi} U e'_j \to \mphi U e'_j \text{ for every } j \in J. $$ 
Let $j \in J$. Since $\mphi U e'_j \in \calh$ then there exists $g \in \ell^2(V)$ such that $\mphi U e'_j = Ug$. In particular, applying Lemma \ref{multipod}(i) 
\begin{equation}\label{lmg} \hphi(m) e'_j= \widehat{\mphi Ue'_j}(m) =\widehat{Ug}(m)= P_{\nul{\slam^*}} L^m g .\end{equation}
Now, for every $n \in \N$, let $g_n \colon V \to \C$ be such that
$$ g_n(u)= \hat{p}_n \big(\big||u| - k_j\big|\big) g(u).$$
Then $g_n \in \ell^2(V)$ for every $n \in \N$. Moreover, for every $m\in\N_0$ and $l\in J$ we have by \eqref{lmg} 
\begin{align*}
    \is{P_{\nul{\slam^*}} L^mg_n}{e'_l}&=\is{g_n}{L^{*m}e'_l}=\is{P_{m+k_l}g_n}{L^{*m}e'_l}=\hat{p}_n\big(|m+k_l-k_j|\big)\is{P_{m+k_l}g}{L^{*m}e'_l}\\
    &=\hat{p}_n\big(|m+k_l-k_j|\big)\is{P_{\nul{\slam^*}} L^mg}{e'_l}=\hat{p}_n\big(|m+k_l-k_j|\big)\hphi(m)\is{e'_j}{e'_l}.
\end{align*}
Thus by Lemma \ref{multipod}(i)
\begin{align*} \widehat{U g_n}(m) = P_{\nul{\slam^*}} L^mg_n=\sum_{l\in J}\is{P_{\nul{\slam^*}} L^mg_n}{e'_l}e'_l=\sum_{l\in J}\hat{p}_n\big(|m+k_l-k_j|\big)\hphi(m)\is{e'_j}{e'_l}e'_l\\
=\hat{p}_n(m)\hphi(m)e'_j = (M_{\hat{p}_n \hphi} Ue'_j)\widehat{\phantom{a}}(m), \quad m \in \N_0. \end{align*}
Hence $Ug_n=M_{\hat{p}_n \hphi} Ue'_j$ for every $n\in\N$. Since $g_n(u)\to g(u)$ and $|g_n(u)|\leqslant|g(u)|$ for every $u\in V$ we deduce by Lebesgue dominated convergence theorem that $g_n\to g$ in $\ell^2(V)$. Thus $$M_{\hat{p}_n \hphi} Ue'_j =Ug_n \to Ug = \mphi U e'_j,$$
which completes the proof.
\end{proof}
Now, we are ready to prove that generalized multipliers whose coefficients are multiple of the identity operator and classical multipliers (see \cite{b-d-p-2015,b-d-p-p-2017}) for a weighted shifts on directed trees are in fact the same object.
\begin{prop}\label{skalarne}
Let $\tcal=(V,E)$ be a countably infinite rooted and leafless directed tree, $\lambdab=\{\lambda_v\}_{v \in V^\circ}\subseteq  (0,\infty)$ and let $\slam\in\bsb(\ell^2(V))$ be left-invertible weighted shift on $\tcal$ with weights $\lambdab$. Let $\{a_n\}_{n=0}^\infty \subset \C$. Define $\hphi \colon \N_0 \to \C$, $\hphi(n) = a_n $ for $n \in \N_0$ and let $\hpsi \colon \N_0 \to \bsb\big(\nul{\slam^*}\big)$, $\hpsi(n) = a_n I_{\nul{\slam^*}}$ for $n \in \N_0$. Then, 
$$ \hphi \text{ is a multiplier for } \slam \text{ if and only if } \hpsi \text{ is a generalized multiplier for }\slam. $$
Moreover, if the above conditions are met, then $\mphi^\lambdab = U^* \mpsi U$.
\end{prop}
\begin{proof}
Let $\hphi$ be a multiplier for $\slam$. Let, for $n \in \N_0$, denote by $\hat{p}_n \colon \N_0 \to \C$  the coefficients of the $n$-th Fejer kernel, i.e., 
$$\hat{p}_n(m) = \left\{ \begin{array}{lc} 1 - \frac{|m|}{n+1} & \text{ if } m \leq n, \\ 0 & \text{ if } m >n. \end{array} \right.$$
Then by \cite[Theorem 4.4(ii)]{b-d-p-2015} the sequence $\hat{p}_n \hphi$ is a multiplier and by \cite[Corollary 3.4]{b-d-p-p-2017} we know that $M_{\hat{p}_n \hphi}^{\lambdab} \tosot \mphi^\lambdab$. On the other hand $M^\lambdab_{\hat{p}_n \hphi}$ is unitarily equivalent to $M_{\hat{p}_n \hpsi}$, i.e., $M^\lambdab_{\hat{p}_n \hphi} = U^*M_{\hat{p}_n \hpsi} U$. Thus $M_{\hat{p}_n \hpsi} \tosot  UM_{ \hphi}^\lambdab U^*$. Since generalized multipliers are SOT-closed (see Corollary \ref{wniosekkom}) we obtain the equality $ UM_{ \hphi}^\lambdab U^* = M_{\hpsi_0}$ for some $\hpsi_0 \colon \N_0 \to \bsb(\nul{\slam^*})$, $\hpsi_0 \in \calm(\slam)$. From Lemma \ref{kosiarka} for every $m \in \N_0$ we get $\hat{p}_n(m) \hpsi(m) \tosot \hpsi_0(m)$. In particular, $\hpsi(m)=\hpsi_0(m)$, since $\hat{p}_n(m)\to 1$ for every $m\in\N_0$.
 Therefore $\hpsi \in \calm(\slam)$ and $\mphi^\lambdab = U^* \mpsi U$.
 
To prove the reverse implication let us assume that $\hpsi\in\multiskal(\slam)$. Then by Proposition \ref{cesaro} $M_{\hat{p}_n \hpsi }\tosot \mpsi$.
Hence $M_{\hat{p}_n \hphi}^\lambdab=U^*M_{\hat{p}_n \hpsi} U\tosot U^*\mpsi U$. Since multipliers are SOT-closed (see \cite[Proposition 3.5.]{b-d-p-p-2017}) $M_{\hat{p}_n \hphi}^\lambdab\tosot U^*\mpsi U = M_{\hphi_0}^\lambdab$ for some $\hphi_0\in\multi$.
This and \eqref{dziobak} imply that for every $m\in \N_0$ we have $p_n(m)\hphi(m)\to \hphi_0(m)$ when $n\to\infty$. Thus $\hphi=\hphi_0\in\multi$.
\end{proof}
On the basis of the above proposition and \cite[Theorem 3.6.]{b-d-p-p-2017} we get the following corollary.
\begin{cor}\label{seminarium}
Assume \eqref{stand2}. If $\slam\in\bsb(\ell^2(V))$ is left-invertible, then $$\{ U^*\mphi U\colon \hphi \in \multiskal(\slam)\} = \mathcal{W}(\slam).$$
\end{cor}

\section{Balanced weighted shifts on directed trees}
In this section, we restrict our attention to the particular case of weighted shifts. This special class of operators admit the Wold-type decomposition. Below, we recall the definition of balanced weighted shifts on directed trees.
\begin{dfn} 
Let $\tcal=(V,E)$ be a countably infinite rooted and leafless directed tree, $\lambdab=\{\lambda_v\}_{v \in V^\circ}\subseteq  (0,\infty)$ and let $\slam$ be a bounded weighted shift on $\tcal$ with weights $\lambdab$. If $$\|\slam e_u\| = \|\slam e_v\| \text{ for every } u, \, v \in V \text{ such that } |u| = |v|,$$ then we say that $\slam$ is \textit{balanced}.
\end{dfn}
It is a large class of operators. The examples can be found in \cite[Section 5]{c-p-t-2} or in Example \ref{4}.  
First, let us note that the orthogonality of two vectors concentrated on one generation (i.e. elements of $\ell^2(V_k)$ for some $k \in \N_0$) is preserved by $\slam$. The proof of this fact is based on the following lemma.
\begin{lem}\label{orto1}
Assume \eqref{stand1}. Let $\slam\in\bsb(\ell^2(V))$ be balanced and let $f \in \ell^2(V_k)$, $g\in \ell^2(V_l)$ for some $k$, $l \in \N_0$. If $u'\in V_{k+n}$, then 
\begin{equation}\label{fsfg}\is{\slam^n f}{ \slam^n g}=\prod_{j=1}^{n}\|\slam e_{\pan{j}{u'}}\|^2\is{f}{g}, \quad n\in\N_0.\end{equation}
\end{lem}
\begin{proof} The equality \eqref{fsfg} is obvious for $n=0$. Assume now that $n=1$. Note that $\slam f\in \ell^2(V_{k+1})$ and $\slam g\in \ell^2(V_{l+1})$, since $f \in \ell^2(V_k)$, $g\in \ell^2(V_l)$. If $l\neq k$, then $ \ell^2(V_{k+1}) \perp \ell^2(V_{l+1})$ and $ \ell^2(V_k) \perp \ell^2(V_l)$. Hence $\is{\slam f}{\slam g}=0$ and $\is{f}{g}=0$. Assume now, that $l=k$. By the definition of the subspace $\ell^2(V_k)$ we have $f=\sum_{u\in V_k}f(u) e_u$ and $g=\sum_{v\in V_k}g(v) e_v$. Since $\slam$ is balanced, $\|\slam e_u\|=\|\slam e_{\pa{u'}}\|$ for every $u\in V_k$. Then the definition of $\slam$ and orthogonality of the family $\{e_u\colon u\in V\}$ implies the following equalities 
\begin{align*}
    \is{\slam f}{\slam g} &= 
    \big\langle \slam \big( \sum_{u \in V_k} f(u) e_u \big), \slam \big( \sum_{v \in V_k} g(v) e_v \big) \big\rangle
    \\&=\big\langle  \sum_{u \in V_k} \sum_{w \in \dzi{u}} \lambda_w f(u) e_w , \sum_{v \in V_k} \sum_{z \in \dzi{v}} \lambda_z g(v) e_z  \big\rangle\\
    &=\sum_{u,v \in V_k}\sum_{w \in \dzi{u}}\sum_{z \in \dzi{v}} \lambda_w f(u) \overline{\lambda_z g(v)} \langle e_w, e_z \rangle
    =\sum_{u \in V_k}\sum_{w \in \dzi{u}} |\lambda_w|^2 f(u) \overline{g(u)}\\
    &=\sum_{u \in V_k} \|\slam e_u\|^2 f(u) \overline{g(u)} 
    =\|\slam e_{\pa{u'}}\|^2\sum_{u \in V_k}  f(u) \overline{g(u)} =\|\slam e_{\pa{u'}}\|^2\is{f}{g}.
\end{align*}
The rest of the claim follows by the inductive argument.
\end{proof}
\begin{cor}
Assume \eqref{stand1}. Let $\slam\in\bsb(\ell^2(V))$ be balanced and let $f \in \ell^2(V_k)$, $g\in \ell^2(V_l)$ for some $k$, $l \in \N_0$. If $f$ and $g$ are orthogonal then
$$\slam^n f\perp \slam^n g, \quad n\in\N_0.$$
\end{cor}
By the above Corollary and \cite[Theorem 6.4]{b-d-p-p-2017} we deduce that balanced weighted shifts are in fact orthogonal sums of classical weighted shifts.
\begin{cor}
Assume \eqref{stand1}. Let $\slam\in\bsb(\ell^2(V))$ be balanced and let $\{e'_j\}_{j\in J}$ be a separated basis of $\nul{\slam^*}$.  Then
$$ \displaystyle \slam = \bigoplus_{j \in J} \slam|_{\overline{\lin\{ \slam^n e'_j \colon n \in \N_0\}}}.$$
\end{cor}

Below, we give the characterization of commutant of balanced weighted shifts with finite dimensional kernel $\mathcal{N}(\slam^*)$. We will need the following notation which allude to \cite{shi}. 
For two sequences $a = \{a_n\}_{n=0}^\infty$, $b = \{b_n\}_{n=0}^\infty\in\C^{\N_0}$ we define its classical Cauchy multiplication $a*b$ by the formula $(a*b)(n) = \sum_{k=0}^n a_k b_{n-k}$, $n\in \N_0$. 
 For given $\beta_1,\beta_2\in (0,+\infty)^{\N_0}$ we set $$H^\infty(\beta_1,\beta_2):=\big\{a\in\C^{\N_0}\colon a*b\in \ell^2(\beta_2) \text{ for every } b\in\ell^2(\beta_1) \big\},$$
 where for any sequence $\beta = \{\beta_n\}_{n=0}^\infty\subset(0,+\infty)$ the symbol $\ell^2(\beta)$ denotes the weighted $\ell^2$ space $\big\{\{a_n\}_{n=0}^\infty\in\C^{\N_0}\colon \sum_{n=0}^\infty |a_n|^2\beta_n<\infty \big\}$. In particular case $H^\infty(\beta):=H^\infty(\beta,\beta)$ if $\beta\in(0,+\infty)^{\N_0}$.

\begin{lem} \label{hinfty}
Assume \eqref{stand1}. Let $\slam\in\bsb(\ell^2(V))$ be balanced and bounded from below by constant $c>0$. Assume that $\{e'_j\}_{j \in J}$ is a separated basis of $\nul{\slam^*}$. Then
\begin{enumerate}
    \item[(i)] for every $i,j\in J$ there exists constants $c_{ij}>0$ and $C_{ij}>0$ such that
   $$ c_{ij} \leq \frac{\|\slam^n e'_i\|}{\|\slam^n e'_j\|} \leq C_{ij},  \quad n\in\N_0.$$
   \item[(ii)] $ H^\infty\Big(\{\|S^ne'_j\|^2\}_{n=0}^\infty,\{\|S^ne'_i\|^2\}_{n=0}^\infty\Big)=H^\infty\Big(\{\|\slam^n e_{\koo}\|^2\}_{n=0}^\infty\Big)$ for every $i,j\in J$.
   
\end{enumerate}
\end{lem}
   \begin{proof}
(i) Let $\{c_n\}_{n=0}^\infty \subset \R$ be a sequence such that $\|\slam e_u\| = c_n$ if $|u|=n$. Let $k_l\in\N_0$ satisfies the condition $e'_l\in\ell^2(V_{k_l})$ for $l\in J$. Then by \eqref{orto1} for every $n >| k_j - k_i|$ we have the estimation
\begin{align*}
   \frac{ \|\slam^n e'_i\|}{\|\slam^n e'_j\|} =\frac{ c_{k_i}\cdots c_{k_i+n-1} }{ c_{k_j}\ldots c_{k_j+n-1}} \leq \bigg(\frac{\|\slam\|}{c}\bigg)^{|k_j-k_i|}.
\end{align*}
This completes the proof since $i$, $j \in J$ are arbitrary.

(ii) It follows from (i), since we can choose $\{e'_j\}_{j \in J}$ such that $e_{\koo}=e'_j$ for some $j \in J$.
\end{proof}

 \begin{rem}\label{balanced} Assume \eqref{stand1}. Suppose that $\slam\in\bsb(\ell^2(V))$ is balanced.
 Then, \cite[Theorem 6.4.]{b-d-p-p-2017} implies that every $f\in\ell^2(V)$ can be decomposed as $\bigoplus_{n=0}^\infty \slam^nf_n$, where $f_n\in\nul{\slam^*}$ for $n\in\N_0$. Thus we have the following characterization of $\calh$.
 Namely, if $\{f_n\}_{n=0}^\infty\subset\nul{\slam^*}$, then
 \begin{align} \label{balanced1}
 \begin{gathered}
    \sum_{n=0}^\infty f_nz^n \text{ is convergent for every } z\in\Dbb_r \text{ and } \ff\in\calh, \text{ where } \ff(z)=\sum_{n=0}^\infty f_nz^n,\, z\in\Dbb_r,\\ \text{ if and only if } \sum_{n=0}^\infty \big\|\slam^n f_n\big\|^2<\infty. 
 \end{gathered}
 \end{align}
 \end{rem}
\begin{thm} \label{kom}
Let $\tcal=(V,E)$ be a countably infinite rooted and leafless directed tree, and $\lambdab=\{\lambda_v\}_{v \in V^\circ}\subseteq  (0,\infty)$. Assume that $\slam\in\bsb(\ell^2(V))$ is balanced and left-invertible. Assume also that 
$\dim \nul{\slam^*}<\infty$ and $\{e'_j\}_{j \in J}$ is a separated basis of $\nul{\slam^*}$. Then $$\calm(\slam)=\Big\{\hphi\colon \N_0\to\bsb\big(\nul{\slam^*}\big)\Big| \big\{\bis{\hphi(n)e'_j}{e'_i}\big\}_{n=0}^\infty\in H^\infty\big(\{\|S^ne_{\koo}\|^2\}_{n=0}^\infty\big), \, i,j\in J \Big\}.$$
\end{thm}
\begin{proof}
 Let $\ff = \sum_{n=0}^\infty \hat{\ff}(n) z^n$ be a vector in $\aa(\nul{\slam^*})$ and let $\alpha_{n,j}=\is{\hat{\ff}(n)}{e'_j}$ for $n\in\N_0$ and $j\in J$. Then
\begin{align*}
    \sum_{j \in J} \sum_{n=0}^\infty  | \alpha_{n,j}|^2 \| \slam^n  e'_j  \|^2 = \sum_{n=0}^\infty \sum_{j \in J} | \alpha_{n,j}|^2 \| \slam^n  e'_j  \|^2 = \sum_{n=0}^\infty \big\| \slam^n \big(\sum_{j \in J} \alpha_{n,j} e'_j \big) \big\|^2= \sum_{n=0}^\infty \big\| \slam^n \hat{\ff}(n) \big\|^2.
\end{align*} Hence, by 
 \eqref{balanced1}
 $$ \ff\in\calh \text{ if and only if } \{\alpha_{n,j}\}_{n=0}^\infty \in \ell^2\big(\{\|\slam^n e'_j\|^2\}_{n=0}^\infty\big) \text{ for every } j\in J.$$

Let $\hphi\colon \N_0\to\bsb(\nul{\slam^*})$ and let $a_{i,j}^{(n)}:=\is{\hphi(n)e'_j}{e'_i}$ for $n\in\N_0$ and $i,j\in J$. Then $\hphi$ can be represented as the matrix: 
$$ \hphi(n) = \Big[ a_{i,j}^{(n)}\Big]_{i,j \in J}, \quad n \in \N_0.$$ Note that
\begin{align} \label{pon}
    \notag \sum_{n=0}^\infty\|\slam^n \big((\hphi* \hat{\ff}) (n)\big) \|^2 &=\sum_{n=0}^\infty \Big\|\slam^n \Big(\sum_{k=0}^n \hphi(k) \hat{\ff}(n-k) \Big)  \Big\|^2  = \sum_{n=0}^\infty\Big\|\slam^n \Big( \sum_{k=0}^n \hphi(k) \sum_{j \in J} \alpha_{n-k,j} e'_j \Big)\Big\|^2 \\
     &=  \sum_{n=0}^\infty \Big\| \sum_{j \in J} \sum_{k=0}^n \alpha_{n-k,j}\slam^n \Big( \hphi(k)e'_j\Big) \Big\|^2=   \sum_{n=0}^\infty \Big\| \sum_{j \in J} \sum_{k=0}^n \sum_{i \in J}  \alpha_{n-k,j} a_{i,j}^{(k)} \slam^n e'_i \Big\|^2\\
 \notag   &=\sum_{n=0}^\infty \sum_{i \in J} \Big\| \sum_{k=0}^n \sum_{j \in J}    \alpha_{n-k,j} a_{i,j}^{(k)} \slam^n e'_i \Big\|^2.
 \end{align}
Assume that $\hphi \in \calm(\slam)$ and fix $i_0, j_0 \in J$. Hence, by Remark \ref{balanced} we have  
 \begin{align*}
 \sum_{n=0}^\infty\sum_{i \in J} \Big\| \sum_{k=0}^n \sum_{j \in J}    \alpha_{n-k,j} a_{i,j}^{(k)} \slam^n e'_i \Big\|^2=\sum_{n=0}^\infty\big\|\slam^n \big((\hphi* \hat{\ff}) (n)\big) \big\|^2=\sum_{n=0}^\infty\big\|\slam^n \big(\widehat{\mphi\ff}(n)\big) \big\|^2<\infty,
 \end{align*}
  for any $\ff\in\calh$. Let $\{\alpha_n\}_{n=0}^\infty\in\ell^2(\|\slam^ne'_{j_0}\|^2)$ and define $\gg(z)=\sum_{n=0}^\infty \alpha_ne'_{j_0}z^n$ for $z\in\Dbb_r$. Then $\gg \in \calh$ by \eqref{balanced1} and
$$\sum_{n=0}^\infty \Big|\sum_{k=0}^n   \alpha_{n-k} a_{i_0,j_0}^{(k)} \Big|^2  \|\slam^n e'_{i_0} \|^2 < \infty,$$
which means that $\big\{\alpha_{n}\big\}_{n=0}^\infty * \big\{a_{i_0,j_0}^{(n)}\big\}_{n=0}^\infty \in \ell^2\big(\{\|\slam^n e'_{i_0}\|^2\}_{n=0}^\infty\big).$
Since $\big\{\alpha_{n}\big\}_{n=0}^\infty$ is arbitrary we obtain that \begin{align*}\big\{a_{i_0,j_0}^{(n)}\big\}_{n=0}^\infty\in H^\infty\big(\{\|S^ne'_{j_0}\|^2\}_{n=0}^\infty,\{\|S^ne'_{i_0}\|^2\}_{n=0}^\infty\big), \end{align*}
which is equivalent, by Lemma \ref{hinfty}(ii), to the condition 
\begin{align*}\big\{a_{i_0,j_0}^{(n)}\big\}_{n=0}^\infty\in H^\infty\big(\{\|S^ne_{\koo}\|^2\}_{n=0}^\infty\big). \end{align*}
Thus the left hand side is contained in the right hand side.

Now, suppose that $\hphi$ satisfies 
$$\big\{a_{ij}^{(n)}\big\}_{n=0}^\infty = \big\{\bis{\hphi(n)e'_j}{e'_i}\big\}_{n=0}^\infty\in H^\infty\Big(\{\|S^n e_{\koo}\|^2\}_{n=0}^\infty \Big), \quad i,j\in J,$$
which, by Lemma \ref{hinfty}(ii) is equivalent to the condition
$$\big\{a_{ij}^{(n)}\big\}_{n=0}^\infty = \big\{\bis{\hphi(n)e'_j}{e'_i}\big\}_{n=0}^\infty\in H^\infty\Big(\{\|S^ne'_j\|^2\}_{n=0}^\infty,\{\|S^ne'_i\|^2\}_{n=0}^\infty\Big), \quad i,j\in J.$$ Let $\ff \in \calh$. Then by \eqref{pon}, the definition of $*$, triangle inequality and the fact that $\{\alpha_{n,j}\}_{n=0}^\infty \in \ell^2\big(\{\|\slam^n e'_j\|^2\}_{n=0}^\infty\big)$ for every $ j \in J$ we obtain
\begin{align*}
    \sum_{n=0}^\infty\Big\|\slam^n \big((\hphi* \hat{\ff}) (n)\Big) \big\|^2&=\sum_{n=0}^\infty \sum_{i \in J} \Big\| \sum_{k=0}^n \sum_{j \in J}    \alpha_{n-k,j} a_{i,j}^{(k)} \slam^n e'_i \Big\|^2\\
     &=\sum_{i \in J} \sum_{n=0}^\infty \Big| \sum_{j \in J}\sum_{k=0}^n     \alpha_{n-k,j} a_{i,j}^{(k)}\Big|^2 \cdot\big\| \slam^n e'_i \big\|^2\\
     &=\sum_{i \in J} \sum_{n=0}^\infty \Big| \sum_{j \in J}    \Big( \big\{ \alpha_{k,j}\big\}_{k=0}^\infty * \big\{ a_{i,j}^{(k)}\big\}_{k=0}^\infty \Big)(n)\Big|^2 \cdot\big\| \slam^n e'_i \big\|^2\\
     &\leqslant 2\sum_{i \in J} \sum_{j \in J}\sum_{n=0}^\infty \Big|     \Big( \big\{ \alpha_{k,j}\big\}_{k=0}^\infty * \big\{ a_{i,j}^{(k)}\big\}_{k=0}^\infty \Big)(n)\Big|^2 \cdot\big\| \slam^n e'_i \big\|^2<\infty.
\end{align*}
 Hence $\hphi \in \calm(\slam)$, which completes the proof.
\end{proof}
This combined with Theorem \ref{komutant} give us the following characterization of commutant of balanced weighted shift with finite dimensional kernel $\nul{\slam^*}$. 
\begin{cor}
If the assumption of Theorem \ref{kom} are satisfied and $U$ is given by \eqref{ufor} with $T=\slam$, then the set  $\{\slam\}'$ - the commutant of $\slam$, can be expressed as $$U\{\slam\}'U^*=\Big\{\mphi\Big|  \hphi\colon \N_0\to\bsb\big(\nul{\slam^*}\big), \, \big\{\bis{\hphi(n)e'_j}{e'_i}\big\}_{n=0}^\infty\in H^\infty\Big(\{\|S^ne_{\koo}\|^2\}_{n=0}^\infty\Big), \, i,j\in J \Big\}. $$ 
\end{cor}
From Theorem \ref{kom} we deduce the following corollary.
\begin{cor} 
Assume \eqref{stand2}. Let $\slam\in\bsb(\ell^2(V))$ be balanced and bounded from below. Assume that $\dim \nul{\slam^*}<\infty$. Then for every $A\in \bsb\big(\nul{\slam^*}\big)$ the sequence $\hphi \colon \N_0 \to \bsb\big(\nul{\slam^*}\big)$ given by the formula $\hphi(n) = \chi_{\{m\}}(n) A$ for some $m \in \N_0$ is a generalized multiplier for $\slam$.
\end{cor}
\begin{proof}
Let $\{e'_j\}_{j\in J}$ be a separated basis of $\nul{\slam^*}$. Fix $i$, $j \in J$ and denote by $a_{i,j} = \{a^{(n)}_{i,j}\}_{n=0}^\infty$ the sequence of complex numbers such that  $a_{i,j}^{(n)}:=\is{\hphi(n)e'_j}{e'_i}$ for $n\in\N_0$. Let $\{\alpha_n\}_{n=0}^\infty \in \ell^2\big(\|\slam^n e_{\koo}\|^2\big)$. Then
\begin{align*}
    \sum_{n=0}^\infty \big| (a_{i,j}* \alpha)(n)\big|^2 \| \slam^n e_{\koo}\|^2&=\sum_{n=m}^\infty \big|a_{i,j}^{(m)}\alpha_{n-m}\big|^2 \| \slam^n e_{\koo}\|^2=|a_{i,j}^{(m)}|^2 \sum_{n=0}^\infty \big|\alpha_n\big|^2 \| \slam^{n+m} e_{\koo}\|^2\\
    &\leq |a_{i,j}^{(m)}|^2  \|\slam^m\|^2 \sum_{n=0}^\infty \big|\alpha_n\big|^2 \| \slam^{n} e_{\koo}\|^2 <\infty.
\end{align*}
This and Theorem \ref{kom} imply that $\hphi\in\calm(\slam)$.
\end{proof}

\section{Reflexivity}
In this section we are going to prove two independent criteria for reflexivity of weighted shifts on directed trees. One of them is a stronger version of \cite[Theorem 4.3]{b-d-p-p-2017}. Let us state two general lemmas (see Section 3 for used notation) for left-invertible, analytic operator $T\in \bsb(\hh)$. The first one shows that all numbers from $\Dbb_r$ are eigenvalues for the operator $\unitT$, which is unitary equivalent to $T$. For the sake of completness we give a proof of this lemma (see also \cite[Proof of Theorem 5.1]{c-t}).
\begin{lem}\label{wlasneT}
Let $T\in\bsb(\hh)$ be left-invertible and analytic. Then 
\begin{enumerate}
\item[(i)] $\unitT^* k_{\calh}(\cdot, \overline{\lambda})e = \lambda k_{\calh}(\cdot, \overline{\lambda})e$ for every $e \in \ee$ and $\lambda \in \Dbb_r$,
\item[(ii)] $\sigma_p(\unitT^*)\supset \Dbb_r$.
\end{enumerate}
\end{lem}
\begin{proof}
(i) Let $\ff\in\calh$, $\lambda\in\Dbb_r$ and $e\in\ee$. By properties \eqref{wlasnosci} and the defintion of $\unitT$ we get the following equalities
\begin{align*}
    \is{\ff}{\unitT^*k_\calh(\cdot,\bar{\lambda})e}_{\calh}=\is{\unitT\ff}{k_\calh(\cdot,\bar{\lambda})e}_{\calh}&=\is{(\unitT\ff)(\bar{\lambda})}{e}_{\ee}\\
    &=\is{\bar{\lambda}\ff(\bar{\lambda})}{e}_{\ee}=\bar{\lambda}\is{\ff(\bar{\lambda})}{e}_{\ee}=\is{\ff}{\lambda k_\calh(\cdot,\bar{\lambda})e}_{\calh},
\end{align*}
which proves (i). 

(ii) It is a direct consequence of (i).
\end{proof}
In the next lemma we consider linear independence of functions  $k_\calh(\cdot,\lambda)e_1$,$\ldots$, $k_\calh(\cdot,\lambda)e_n\in\calh$, which will be used in the proof of the first criterion for reflexivity.
\begin{lem}\label{nzal1} Let $T\in\bsb(\hh)$ be left-invertible and analytic.  
Assume that $e_1,\ldots,e_n\in\ee$ are linearly independent. Then the functions $k_\calh(\cdot,\lambda)e_1$,$\ldots$, $k_\calh(\cdot,\lambda)e_n$ are
linearly independent in $\calh$, where $\lambda\in\Dbb_r$.
\end{lem}
\begin{proof} Fix $\lambda\in\Dbb_r$ and
assume that
\begin{equation*}
\sum_{j=1}^n\alpha_jk_\calh(\cdot,\lambda)e_j=0, \text{ for some } \alpha_1,\ldots,\alpha_n\in\C.
\end{equation*}
 Let $\ff(z)=\sum_{j=1}^n\alpha_je_j$ for every $z\in\Dbb_r$. Then $\ff\in\calh$, since it is a constant function.
 Moreover, we have
 \begin{align*}
\begin{aligned}
    0=\Bis{\ff}{\sum_{j=1}^n\alpha_jk_\calh(\cdot,\lambda)e_j}_\calh&=\sum_{j=1}^n\overline{\alpha_j}\is{\ff}{k_\calh(\cdot,\lambda)e_j}_\calh\\
    &=\sum_{j=1}^n\overline{\alpha_j}\is{\ff(\lambda)}{e_j}_\ee=\Bis{\ff(\lambda)}{\sum_{j=1}^n\alpha_je_j}_\ee=\Big\|\sum_{j=1}^n\alpha_je_j\Big\|^2, \quad \ff\in\calh.
\end{aligned}
\end{align*}
 This and the linear independence of $e_1,\ldots,e_n$ imply $\alpha_1=\ldots=\alpha_n=0$, which completes the proof.  
\end{proof}
Now, we are going to prove two criteria, the main results of this section. First of them is based on the form of the generalized multipliers, whose coefficients are multiple of the identity operator (see Proposition \ref{skalarne}). Let us notice that the criterion for reflexivity of left-invertible analytic operator was stated in \cite[Theorem 4.1 and Corollary 4.5]{c-p-t-2}. Contrary to the mentioned criterion we do not assume that the operator is polynomially bounded and that the spectral radius of the Cauchy dual is not greater than 1, which are quite restrictive (consider the unilateral weighted shift multiplied by a constant).
\begin{thm} \label{T-ref}
Let $T\in\bsb(\hh)$ be left-invertible and analytic. Assume that $$\big\{ U^*\mphi U \colon \hphi \in \multiskal(T)\big\}=\mathcal{W}(T).$$ Then $T$ is reflexive.
\end{thm}
\begin{proof} Let $A \in \Alg \Lat T$, then
$\mathcal{A}^* \in \Alg \Lat\unitT^*$, where $\mathcal{A}=UAU^*$. By Lemma \ref{wlasneT}(i) there exists a function $\varphi \colon \Dbb_r \times \ee \to \C$ such that
\begin{align}\label{akalambdae}
    \cala^* k_{\calh}(\cdot, \overline{\lambda})e = \varphi(\lambda,e)k_{\calh}(\cdot, \overline{\lambda})e. 
\end{align}
Hence Lemma \ref{wlasneT}(i) implies that 
\begin{align}\label{malaprzem}
   \unitT^* \cala^* k_{\calh}(\cdot, \overline{\lambda})e = \varphi(\lambda,e)\lambda k_{\calh}(\cdot, \overline{\lambda})e= \cala^*\unitT^* k_{\calh}(\cdot, \overline{\lambda})e. 
\end{align}
By \eqref{wlasnosci} the set $\{k_{\calh}(\cdot, \overline{\lambda})e\colon \lambda\in\Dbb_r, e\in\ee\}$ is linearly dense in $\calh$. This and equation \eqref{malaprzem} leads to the equality $\unitT \cala=\cala\unitT$. Hence by Theorem \ref{komutant} 
\begin{equation}\label{Amnoznik} \cala=M_{\hpsi} \text{ for some } \hpsi\in \calm(T).\end{equation} 

Let $k_{\lambda,e}:=k_\calh(\cdot,\bar{\lambda})e$, where $\lambda\in\Dbb_r$, $e\in\ee$. We are going to show that the value of $\varphi(\lambda,e)$ does not depend
on $e\in\ee\setminus\{0\}$. 
First, we assume that $\dim\ee=1$. Let $e\in\ee\setminus\{0\}$ and $\alpha\in\C\setminus\{0\}$. It is obvious
that $k_{\lambda,\alpha e}=\alpha k_{\lambda,e}$. Hence by \eqref{akalambdae}, we have
\begin{align*}
  \varphi(\lambda,\alpha e)k_{\lambda,e}=\alpha^{-1} \varphi(\lambda,\alpha e)k_{\lambda,\alpha e}=\alpha^{-1} \cala^*k_{\lambda,\alpha e}=\cala^*k_{\lambda,e}=\varphi(\lambda,e)k_{\lambda,e}.
\end{align*}
Therefore $ \varphi(\lambda,\alpha e)=\varphi(\lambda,e)$, since $k_{\lambda,e}\neq0$, by Lemma \ref{nzal1}.

Now, let $\dim\ee\geqslant2$. Choose $e,e'\in\ee$ such that $e$ and $e'$ are linearly independent. Applying \eqref{akalambdae} once more we obtain
\begin{align*}
\varphi(\lambda,e)k_{\lambda,e}+\varphi(\lambda,e')k_{\lambda,e'}&= \cala^*k_{\lambda,e}+\cala^*k_{\lambda,e'}=\cala^*k_{\lambda,e+e'}\\&=\varphi(\lambda,e+e')k_{\lambda,e+e'}   
=\varphi(\lambda,e+e')k_{\lambda,e}+\varphi(\lambda,e+e')k_{\lambda,e'}.  
\end{align*}
Hence the linear independence of $k_{\lambda,e}$ and $k_{\lambda,e'}$ implies that $\varphi(\lambda,e)=\varphi(\lambda,e+e')=\varphi(\lambda,e')$.
If $e,e'\in\ee\setminus\{0\}$ are linearly dependent, then we repeat the argument for one-dimensional case.  

Since the values of $\varphi(\lambda,e)$ does not depend on $e\in\ee\setminus\{0\}$ we may define the function $\varphi_0\colon\Dbb_r\to\C$ such that $\varphi_0(\lambda)=\overline{\varphi(\overline{\lambda},e)}$ for every $\lambda\in\Dbb_r$ and $e\in\ee\setminus\{0\}$. Thus by \eqref{wlasnosci} and \eqref{akalambdae} we have the following equalities
\begin{align*}
  \is{(\cala\ff)(\lambda)}{e}_\ee&=\is{\cala\ff}{k_{\bar\lambda,e}}_\calh=\is{\ff}{\cala^*k_{\bar\lambda,e}}_\calh=\is{\ff}{\varphi(\bar\lambda,e)k_{\bar\lambda,e}}_\calh\\&=\is{\overline{\varphi(\overline{\lambda},e)}\ff(\lambda)}{e}_\ee=\is{\varphi_0(\lambda)\ff(\lambda)}{e}_\ee, \quad   \ff\in\calh, \, \lambda\in\Dbb_r, \, e\in\ee. 
\end{align*}
Now, we apply the above equality to function $\ff=Ue$ and by \eqref{Amnoznik} and Lemma 1 we obtain
\begin{align*}
   \|e\|^2\varphi_0(\lambda)&=\bis{\varphi_0(\lambda)(Ue)(\lambda)}{e}_\ee=\bis{(AUe)(\lambda)}{e}_\ee\\
   &=\bis{(M_{\hpsi}Ue)(\lambda)}{e}_\ee=\Bis{\sum_{n=0}^\infty\hpsi(n)e\lambda^n}{e}_\ee=\sum_{n=0}^\infty\is{\hpsi(n)e}{e}_\ee\lambda^n, \lambda\in\Dbb_r.
\end{align*}
Therefore $\varphi_0$ is analytic in $\Dbb_r$. Let $\widehat{\varphi_0}\colon\N_0\to\C$ be the sequence of the coefficients of $\varphi_0$ i.e. $\varphi_0(\lambda)=\sum_{n=0}^\infty\widehat{\varphi_0}(n)\lambda^n$ for every $\lambda\in\Dbb_r$. In particular, $\widehat{\varphi_0}(n)=\frac{\is{\hpsi(n)e}{e}_\ee}{\|e\|^2}$ for every $e\in\ee$ and $n\in\N_0$. Thus $\hpsi(n)=\widehat{\varphi_0}(n)I_\ee$ for every $n\in\N_0$, which means that $\hpsi\in \multiskal(T)$.
Hence, by \eqref{Amnoznik} and the assumption of the theorem $$A=U^*\cala U= U^*M_{\hpsi} U\in\mathcal{W}(T),$$ which implies the reflexivity of $T$. 
\end{proof}
In case of left-invertible weighted shift operator on directed tree the assumption of Theorem \ref{T-ref} is automatically satisfied by Corollary \ref{seminarium}. Hence, any such operator is reflexive.
\begin{cor}\label{S-ref}
Assume \eqref{stand2}. If $\slam\in\bsb(\ell^2(V))$ is left-invertible, then $\slam$ is reflexive.
\end{cor}
Now, we are going to prove the second criterion for reflexivity of weighted shift operator, which is a generalization of \cite[Theorem 4.3]{b-d-p-p-2017}. Note that we do not assume left-invertibility of the operator. 
\begin{thm} \label{reflex}
Suppose that $\tcal=(V,E)$ is a countably infinite rooted and leafless directed tree,  and $\lambdab=\{\lambda_v\}_{v \in V^\circ}\subseteq  (0,\infty)$. Let $\pth =(V_\pth, E_\pth) \in \pp$ and $\lambdab_\pth=\{\lambda_v\}_{v\in V^\circ_\pth}$. If $\slam\in\bsb(\ell^2(V))$ and $S_{\lambdab_\pth}\in\bsb(\ell^2(V_\pth))$ is reflexive, then $\slam$ is reflexive.
\end{thm}
\begin{proof}
Suppose that $A\in \Alg \Lat \slam$. Then  $A^*\in \Alg \Lat \slam^*$. 
Careful inspection of \cite[Lemma 4.2(ii)]{b-d-p-p-2017} shows that $A^*|_{\ell^2(V_\pth)}\in \Alg \Lat (\slam^*|_{\ell^2(V_\pth)})$ and
\begin{align}\label{herbata}
    A^*|_{\ell^2(V_\pth)} = \big(M_{\hat \varphi}^{\lambdab_\pth}\big)^{*} \text{ with some } \hat \varphi \in\mathcal{M}(\lambdab_\pth).
\end{align}

Let $v \in V$ and let $u \in \pth$ be such that $|u|=|v|$. Define $n:=|v|$. Let $\mathcal{L}_{uv}$ (resp. $\mathcal{L}_v$) denote the invariant closed subspace for $\slam^*$ generated by $f=\lambda_{\koo|v}e_u-\lambda_{\koo|u}e_v$ (resp. generated by $e_v$). We are going to determine the coefficients $a_{wt}:=\is{A^*e_w}{e_t}$ for $w,t\in V$.

We say that vertex $t$ is an {\it ancestor} of $w$ if $ w \in \des{t}$. We denote this relation by $t \prec w$.  \cite[Equality (4)]{b-d-p-p-2017} gives value of $M_{\hat \varphi}^{\lambdab_\pth}$ on the basis, hence one can show that 
$$ \big(M_{\hat \varphi}^{\lambdab_\pth}\big)^{*} e_w = \sum_{t:t \prec w} \lambda_{t|w} \overline{ \hphi(|w|-|t|)} e_t, \quad w \in V_\pth.$$
First, we observe that $A^*e_v\in \mathcal{L}_v$ , since $\mathcal{L}_v\in \Lat \slam^*\subset \Lat A^*$ and $e_v\in \mathcal{L}_v$. Thus $a_{vw}=0$ if $w\nprec v$. As a consequence, by \eqref{herbata} we get
\begin{align} \label{kapusniak}
    A^*f=\lambda_{\koo|v}\sum_{t:t\prec u}\lambda_{t|u}\overline{\hat{\varphi}(|u|-|t|)}e_t-\lambda_{\koo|u}\sum_{w:w\prec v}a_{vw}e_w.
\end{align}
On the other hand, if $k\in\N_0$ and $k\leq n$, then
\begin{align}\label{nalesniki}
    \slam^{*k}f=\lambda_{\koo|v}\lambda_{\pan{k}{u}|u}e_{\pan{k}{u}}-\lambda_{\koo|u}\lambda_{\pan{k}{v}|v}e_{\pan{k}{v}}.
\end{align}
In particular $\slam^{*n}f=0$. Since $\Lat \slam^*\subset \Lat A^*$ we obtain that $\mathcal{L}_{uv}\in\Lat A^*$. Thus 
\begin{align}\label{pierogi}
    A^*f=\alpha_0f + \alpha_1 \slam^*f + \ldots + \alpha_{n-1} \slam^{*(n-1)} f,
\end{align}
for some $\{\alpha_k\}_{k=0}^{n-1} \subset \C$.

Let $w\prec v$ and let $t \in \pth$ be such that $|t| = |w|=n-k$ for some $k \in \{0,\ldots,n\}$. Now, let us consider two cases:
\begin{itemize}
    \item $t =w$ \newline
     Comparing coefficients at  $e_w$ in formulas \eqref{kapusniak} and \eqref{pierogi} and using \eqref{nalesniki}, we get
        $$ \lambda_{\koo|v} \lambda_{t|u} \overline{\hphi(k)} - \lambda_{\koo|u} a_{vw} = \alpha_k (\lambda_{\koo|v} \lambda_{t|u} - \lambda_{\koo|u} \lambda_{w|v}).$$
        Since, the right hand side is equal to 0 and $\lambda_s\neq0$ for $s\in V^\circ$, we obtain $$a_{vw} = \lambda_{w|v} \overline{\hphi(|v|-|w|)}.$$
    \item $t\neq w$\newline
    Using \eqref{kapusniak}, \eqref{nalesniki}, and \eqref{pierogi} once more and comparing coefficients at $e_t$ and $e_w$ respectively we get
    \begin{align*}
        \lambda_{\koo|v} \lambda_{t|u} \overline{\hphi(k)} = \alpha_k \lambda_{\koo|v} \lambda_{t|u}\quad  \text{ and }  \quad
        \lambda_{\koo|u} a_{vw} = \alpha_k \lambda_{\koo|u} \lambda_{w|v}.
    \end{align*}
    Hence
    $$ \alpha_k = \overline{\hphi(k)} \quad  \text{ and } \quad a_{vw} = \alpha_k \lambda_{w|v} = \lambda_{w|v} \overline{\hphi(|v|-|w|)}.$$
\end{itemize}
Summarizing, for every $v\in V$
\begin{align*}
    \is{A^*e_v}{e_w} = a_{vw}= \left\{ \begin{array}{lc} \lambda_{w|v} \overline{\hphi(|v|-|w|)}& \text{if } w\prec v,\\
    0 & \text{if } w\nprec v.\end{array} \right. .
\end{align*} 

Now, we define the mapping $\gphi^{\lambdab}$  given by the formula \eqref{multiKL1}. We will show that $\hphi$ is a multiplier for $\slam$ (consider the whole tree). 
First, we will show that $\{e_v\colon v\in V\}\subset \dz{M_{\hat \varphi}^\lambdab}$.  For $v\in V$ we have
\begin{align*}
   \sum_{w\in V}\left|\left(\gphi^{\lambdab} e_v\right)(w)\right|^2 =\sum_{w: v \prec w} |\lambda_{v|w} \hphi(|w|-|v|)|^2 = \sum_{w \in V} |\langle e_v,A^*e_w\rangle|^2= \sum_{w \in V} |\langle Ae_v,e_w\rangle|^2 = \|A e_v \|^2 .
\end{align*}
Hence $e_v\in\dz{M_{\hat \varphi}^\lambdab}$ for every $v\in V$. Moreover by \eqref{dziobak}
\begin{align*}
    \is{Ae_v}{e_w}= \left\{ \begin{array}{lc} \lambda_{v|w}\hphi(|w|-|v|)& \text{if } v\prec w,\\
    0 & \text{if } v\nprec w.\end{array} \right. = \langle M_{\hat \varphi}^\lambdab e_v, e_w\rangle.
\end{align*}
Thus we obtain that $A|_{\lin\{e_v\colon v\in V\}}\subset M_{\hat \varphi}^\lambdab$. Since $A\in\bsb(\ell^2(V))$ and $M_{\hat \varphi}^\lambdab$ is closed, we get the equality $A=M_{\hat \varphi}^\lambdab$. Thus
$A\in \mathcal{W}(\slam)$ by \cite[Theorem 3.6.]{b-d-p-p-2017}, which completes the proof.
\end{proof}
Next two examples show that two criteria for reflexivity of weighted shift on directed tree - Theorem \ref{reflex} and Corollary \ref{S-ref} - are independent. 
 \begin{exa}
Let $\tcal_{2} = (V_{2},{E_{2}})$ be the directed tree as in the Example \ref{1alfa}. Define $\pth_1=(V_1,E_1)$ where
\begin{align*}
    V_{1} &= \big\{(0,0) \big\} \cup \big\{ (1,j) \colon \ j \in \N \big\}, \\
    E_{1} &= \Big\{ \big((0,0),(1,1)\big)\Big\}\cup \Big\{ \big((1,j),(1,j+1)\big) \colon\ j \in \N \Big\}.
\end{align*} Let $\slam$ be a weighted shift on $\tcal_{2}$ with weights $\lambdab = \{ \lambda_v\}_{v \in V_{2}^\circ}$ given by
\begin{align*}
\lambda_{(i,j)} 
= \left\{ 
\begin{array}{cl} 
1 & \text{ for } i=1 \text{ and } j\in\N, \\ 
\frac{1}{2^j} & \text{ for } i=2 \text{ and } j\in\N.
\end{array} 
\right.
\end{align*}
Then $\slam\in\bsb(\ell^2(V))$ is not left-invertible since $\lim\limits_{j\to\infty}\slam e_{2j}=0$. Hence, the assumption of Corollary \ref{S-ref} is not satisfied. On the other hand, $\slam|_{\pth_1}$ is an isometric unilateral weighted shift and thus it is reflexive. By Theorem \ref{reflex} the whole operator $\slam$ is reflexive by \cite{Sa}.
\end{exa}
Now, let us define a directed tree, which grows rapidly, i.e., the number of children is multiplied by 4 in next generation. More precisely, each vertex in $n$-th generation has $2^{2n+2}$ children and the number of vertexes in $n$-th generation equals to $2^{(n+1)n}$.
\begin{dfn}
Let $\tcal_4=(V_4,E_4)$ be a rooted directed tree (see Figure \ref{szybkiedrzewo}) such that
\begin{align*}
    V_4 &= \Big\{ (m,n) \in \N_0^2 \colon n \leq 2^{m(m+1)} -1\Big\},\\
    E_4 &= \Big\{ \big((m, n), (m+1, k) \big) \colon  n2^{2m+2}\leq  k \leq (n+1)2^{2m+2}-1,\, m, n, k \in \N_0 \Big\}.
\end{align*}

\end{dfn}

\begin{figure}[ht] 
\begin{tikzpicture}[scale=0.8,transform shape,edge from parent/.style={draw,-latex}, level 1/.append style={level distance=1.5cm,sibling distance=4.5cm}, level 2/.append style={level distance=1.5cm,sibling distance=0.23cm}]
\node[circle,fill=gray!30] {\scriptsize{$(0,0)$}}
child {node[circle,fill=gray!30] {\scriptsize{$(1,0)$}} 
    child {
    edge from parent[dashed] node[above left,font=\scriptsize]{$\tfrac 14$
    }}
    child { edge from parent[dashed] node[left,font=\scriptsize]{}}
    child { edge from parent[dashed] node[left,font=\scriptsize]{}}
    child { edge from parent[dashed] node[left,font=\scriptsize]{}}
    child { edge from parent[dashed] node[left,font=\scriptsize]{}}
    child { edge from parent[dashed] node[left,font=\scriptsize]{}}
    child { edge from parent[dashed] node[left,font=\scriptsize]{}}
    child { edge from parent[dashed] node[left,font=\scriptsize]{}}
    child { edge from parent[dashed] node[left,font=\scriptsize]{}}
    child { edge from parent[dashed] node[left,font=\scriptsize]{}}
    child { edge from parent[dashed] node[left,font=\scriptsize]{}}
    child { edge from parent[dashed] node[left,font=\scriptsize]{}}
    child { edge from parent[dashed] node[left,font=\scriptsize]{}}
    child { edge from parent[dashed] node[left,font=\scriptsize]{}}
    child { edge from parent[dashed] node[left,font=\scriptsize]{}}
    child { edge from parent[dashed] node[above right,font=\scriptsize]{$\tfrac 14$}
    }
     edge from parent node[ above left,font=\scriptsize]{$\tfrac 12$}  
   }
child {node[circle,fill=gray!30] {\scriptsize{$(1,1)$}}
    child {
    edge from parent[dashed] node[above left,font=\scriptsize]{$\tfrac 14$
    }}
    child { edge from parent[dashed] node[left,font=\scriptsize]{}}
    child { edge from parent[dashed] node[left,font=\scriptsize]{}}
    child { edge from parent[dashed] node[left,font=\scriptsize]{}}
    child { edge from parent[dashed] node[left,font=\scriptsize]{}}
    child { edge from parent[dashed] node[left,font=\scriptsize]{}}
    child { edge from parent[dashed] node[left,font=\scriptsize]{}}
    child { edge from parent[dashed] node[left,font=\scriptsize]{}}
    child { edge from parent[dashed] node[left,font=\scriptsize]{}}
    child { edge from parent[dashed] node[left,font=\scriptsize]{}}
    child { edge from parent[dashed] node[left,font=\scriptsize]{}}
    child { edge from parent[dashed] node[left,font=\scriptsize]{}}
    child { edge from parent[dashed] node[left,font=\scriptsize]{}}
    child { edge from parent[dashed] node[left,font=\scriptsize]{}}
    child { edge from parent[dashed] node[left,font=\scriptsize]{}}
    child { edge from parent[dashed] node[above right,font=\scriptsize]{$\tfrac 14$}
    }
     edge from parent node[right = 5pt,font=\scriptsize]{$\tfrac 12$}  
   }
child {node[circle,fill=gray!30] {\scriptsize{$(1,2)$}}
    child {
    edge from parent[dashed] node[above left,font=\scriptsize]{$\tfrac 14$
    }}
    child { edge from parent[dashed] node[left,font=\scriptsize]{}}
    child { edge from parent[dashed] node[left,font=\scriptsize]{}}
    child { edge from parent[dashed] node[left,font=\scriptsize]{}}
    child { edge from parent[dashed] node[left,font=\scriptsize]{}}
    child { edge from parent[dashed] node[left,font=\scriptsize]{}}
    child { edge from parent[dashed] node[left,font=\scriptsize]{}}
    child { edge from parent[dashed] node[left,font=\scriptsize]{}}
    child { edge from parent[dashed] node[left,font=\scriptsize]{}}
    child { edge from parent[dashed] node[left,font=\scriptsize]{}}
    child { edge from parent[dashed] node[left,font=\scriptsize]{}}
    child { edge from parent[dashed] node[left,font=\scriptsize]{}}
    child { edge from parent[dashed] node[left,font=\scriptsize]{}}
    child { edge from parent[dashed] node[left,font=\scriptsize]{}}
    child { edge from parent[dashed] node[left,font=\scriptsize]{}}
    child { edge from parent[dashed] node[above right,font=\scriptsize]{$\tfrac 14$}
    }
     edge from parent node[left=5pt,font=\scriptsize]{$\tfrac 12$}  
   }
child {node[circle,fill=gray!30] {$(1,3)$}
    child {
    edge from parent[dashed] node[above left,font=\scriptsize]{$\tfrac 14$
    }}
    child { edge from parent[dashed] node[left,font=\scriptsize]{}}
    child { edge from parent[dashed] node[left,font=\scriptsize]{}}
    child { edge from parent[dashed] node[left,font=\scriptsize]{}}
    child { edge from parent[dashed] node[left,font=\scriptsize]{}}
    child { edge from parent[dashed] node[left,font=\scriptsize]{}}
    child { edge from parent[dashed] node[left,font=\scriptsize]{}}
    child { edge from parent[dashed] node[left,font=\scriptsize]{}}
    child { edge from parent[dashed] node[left,font=\scriptsize]{}}
    child { edge from parent[dashed] node[left,font=\scriptsize]{}}
    child { edge from parent[dashed] node[left,font=\scriptsize]{}}
    child { edge from parent[dashed] node[left,font=\scriptsize]{}}
    child { edge from parent[dashed] node[left,font=\scriptsize]{}}
    child { edge from parent[dashed] node[left,font=\scriptsize]{}}
    child { edge from parent[dashed] node[left,font=\scriptsize]{}}
    child { edge from parent[dashed] node[above right,font=\scriptsize]{$\tfrac 14$}
    }
     edge from parent node[above right,font=\scriptsize]{$\tfrac 12$}  
   };
\end{tikzpicture}
\caption{}
\label{szybkiedrzewo}
\end{figure}
\begin{exa}\label{4}
Let $\slam$ be a weighted shift on a directed tree $\tcal_4 = (V_4, E_4)$ with weights $\lambdab =\{ \lambda_v\}_{v \in V_4^\circ} \subset \C$ such that $
    \lambda_v = 2^{-|v|}, v \in V_4^\circ.$ Then the operator $S_{\lambdab_\pth}$ is not reflexive for every $\pth \in \pp$ (see \cite[Corolllary 1, pp.105]{shi}). Hence, the assumption of Theorem \ref{reflex} is not satisfied. On the other hand, $\slam$ is an isometry (is left-invertible) and by Corollary \ref{S-ref} it is reflexive.
    
\end{exa}
\section*{Acknowledgments}
All authors were supported by the Ministry of Science and Higher Education of the Republic of Poland.
\bibliographystyle{amsalpha}

\end{document}